\documentclass[12pt]{amsart}
\usepackage{amscd,amssymb,graphicx,url}


\newtheorem{thm}{Theorem}
\newtheorem{cor}[thm]{Corollary}
\newtheorem{lemma}[thm]{Lemma}
\newtheorem{remark}[thm]{Remark}
\newcommand{\sections}{\renewcommand{\thethm}{\thesection.\arabic{thm}}
           \setcounter{thm}{0}}
\newcommand{\nosubsections}{\renewcommand{\thethm}{\thesection.\arabic{thm}}
           \setcounter{thm}{0}}

\newcommand{\co}{\colon}

\newcommand{\bbC}{\mathbb{C}}
\newcommand{\bbH}{\mathbb{H}}
\newcommand{\bbQ}{\mathbb{Q}}
\newcommand{\bbR}{\mathbb{R}}
\newcommand{\bbZ}{\mathbb{Z}}
\newcommand{\cN}{\mathcal{N}}
\newcommand{\cT}{\mathcal{T}}
\newcommand{\za}{\alpha}
\newcommand{\zb}{\beta}
\newcommand{\zd}{\delta}
\newcommand{\zf}{\phi}
\newcommand{\zg}{\gamma}
\newcommand{\zh}{\eta}
\newcommand{\zi}{\iota}
\newcommand{\zj}{\psi}
\newcommand{\zl}{\lambda}
\newcommand{\zm}{\mu}
\newcommand{\zr}{\rho}
\newcommand{\zs}{\sigma}
\newcommand{\zt}{\tau}
\newcommand{\zv}{\varphi}
\newcommand{\zF}{\Phi}
\newcommand{\zG}{\Gamma}
\newcommand{\zL}{\Lambda}

\begin{document}
\title{NET map slope functions }
\author{Walter Parry}
\date{\today}
\address{Department of Mathematics and Statistics\\
Eastern Michigan University\\
Ypsilanti, MI  48197}
\email{walter.parry@emich.edu}

\begin{abstract} This paper studies NET map slope functions.  It
establishes Lipschitz-type conditions for them.  It relates
Lipschitz-type conditions to the half-space theorem.  It gives bounds
on the number of slope function fixed points.  It provides examples of
rational NET maps with many matings.
\end{abstract}

\maketitle

\sections

\section{Introduction}\label{sec:intro}\nosubsections

This paper is part of a series of papers \cite{cfpp}, \cite{fkklpps},
\cite{fpp1}, \cite{fpp2}, \cite{fpp3} on NET maps.  A nearly Euclidean
Thurston (NET) map is a Thurston map $f$ with exactly four
postcritical points such that the local degree of $f$ at each of its
critical points is 2.  These are the simplest Thurston maps with
nontrivial Teichm\"{u}ller spaces.  They are quite tractable and yet
exhibit a wide range of behavior.

Although the emphasis is on NET maps, much of this paper applies to
all Thurston maps with four postcritical points.  Let $f\co S^2\to
S^2$ be a Thurston map with postcritical set $P_f$ containing exactly
four points.

Homotopy classes of simple closed curves in $S^2\setminus P_f$
correspond to slopes, elements of $\overline{\bbQ}=\bbQ\cup \{\infty
\}$.  We enlarge the set $\overline{\bbQ}$ by adjoining one more
element $\odot $, called the nonslope.  Loosely speaking, the slope
function $\zm_f\co \overline{\bbQ}\to \overline{\bbQ}\cup \{\odot \}$
is defined (in Section~\ref{sec:slopefns}) by starting with a slope,
choosing a simple closed curve $\zg$ with slope $s$ and defining
$\zm_f(s)$ to be the slope of a simple closed curve in $f^{-1}(\zg)$.
The multiplier of $s$ is the sum of the reciprocals of the degrees of
the restrictions of $f$ to the connected components of $f^{-1}(\zg)$
which are neither null homotopic nor peripheral.  Much of the interest
in $\zm_f$ stems from the fact that W. Thurston's characterization of
rational maps implies in this situation that if the orbifold of $f$ is
hyperbolic, then $f$ is equivalent to a rational map if and only if
$\zm_f$ has no fixed point with multiplier at least 1.  Slopes of
mating equators are also fixed points of $\zm_f$.

There is an intersection pairing on simple closed curves and a
corresponding intersection pairing on slopes.  In
Section~\ref{sec:genllip} we establish a Lipschitz-type condition
relative to the latter intersection pairing for slope functions of
Thurston maps with exactly four postcritical points.  In
Section~\ref{sec:nmlip} we substantially improve this Lipschitz-type
condition in the case of NET maps.  This result is used in \cite{fpp3}
as well as later in this paper.

We view the Teichm\"{u}ller space associated to $(S^2,P_f)$ as the
upper half-plane $\bbH$.  Using $f$ to pull back complex structures on
$(S^2,P_f)$ induces a pullback map $\zs_f\co \bbH\to \bbH$.  Selinger
proved in \cite{S} that $\zs_f$ extends continuously to the augmented
Teichm\"{u}ller space $\bbH\cup \overline{\bbQ}$.  If $s$ is a slope
such that $\zm_f(s)$ is a slope, not $\odot $, then
$\zs_f(-1/s)=-1/\zm_f(s)$.  In this way a slope $s$ corresponds to a
cusp $t=-1/s$.

Let $s$ be a slope such that $\zm_f(s)\notin \{s,\odot \}$.  Then the
half-space theorem of \cite{cfpp} provides an explicit open interval
about $-1/s$ in $\partial \bbH$, called an excluded interval, which
does not contain the negative reciprocal of the slope of a Thurston
obstruction for $f$.  The proof is based on an argument involving
extremal lengths of families of simple closed curves.
Theorem~\ref{thm:halfsp} gives another proof of this based on a slope
function Lipschitz-type condition.  Theorem~\ref{thm:nmhalfsp}, the
NET map half-space theorem, improves on this for the special case in
which $f$ is a NET map.  Sometimes finitely many of these excluded
intervals cover $\partial \bbH$, allowing one to conclude that $f$ is
rational.  This occurs in Example 6.8 of \cite{cfpp} as well as for
Douady's rabbit \cite[Figure 4]{fkklpps}.  However, after the
half-space theorem was first proved, computations suggested that there
exist examples of rational NET maps for which every finite union of
excluded intervals which arise from the half-space theorem omits an
open interval of $\partial \bbH$.  Theorem~\ref{thm:omit2} allows for
easy construction of many such examples, the first of which we are
aware.  See Remark~\ref{remark:omit2}.  This possibility led to the
extended half-space theorem, a qualitative version of which is
discussed in \cite[Section 10]{fkklpps}.  The extended half-space
theorem provides an explicit open interval about an initial cusp
$t=-1/s$ such that $\zm_f(s)\in \{s,\odot \}$.  This open interval,
except for $t$, is a union of infinitely many excluded intervals which
arise from the half-space theorem.  In a forthcoming paper \cite{p} we
prove for Thurston maps with four postcritical points having no
obstructions with multipliers equal to 1 that finitely many excluded
intervals arising from both the half-space and extended half-space
theorems cover $\partial \bbH $.  Thus in this case the union of all
excluded intervals arising from the half-space theorem is a cofinite
subset of $\partial \bbH$ (whose complement consists of elements of
$\overline{\bbQ}$).  However, Theorem~\ref{thm:omit3} allows one to
easily construct NET maps having obstructions with multiplier 1 for
which the union of all excluded intervals arising from the half-space
theorem is not a cofinite subset of $\partial \bbH$.  See
Remark~\ref{remark:omit3}.  These are the first such examples of which
we are aware.

If $f$ has an obstruction with multiplier 1, then $f$ commutes with a
(nonzero power of a primitive) Dehn twist $\zt$ about the obstruction
up to homotopy.  So if $\zm_f$ fixes a slope $s$ other than the slope
of the obstruction, then $\zm_f$ fixes every slope in the orbit of $s$
under the action of the infinite cyclic group generated by $\zt$.  In
this way it is possible for $\zm_f$ to have infinitely many fixed
points.  This occurs for the NET map presented in
Figure~\ref{fig:omit3}.  (See Remark~\ref{remark:omit3} and the proof
of Theorem~\ref{thm:omit3}.)  On the other hand, if $f$ does not have
an obstruction with multiplier 1, then Theorem~\ref{thm:fixedpts}
gives a bound on the number of slopes which $\zm_f$ can fix.  In
effect, this bounds the number of slope function fixed points in terms
of the degree of $f$.

Finally, Section~\ref{sec:matings} deals with polynomial matings.  See
\cite{bekmprt} for a thorough survey of this topic.
Section~\ref{sec:matings} presents an infinite sequence of rational
NET maps $f_n$ with degree $n$ such that $f_n$ has at least
$\left\lceil(n-1)/2\right\rceil $ formal matings.
Remark~\ref{remark:matings} shows how to modify these maps to obtain
rational Thurston maps with many topological matings.

The paper \cite{fpp3} also deals somewhat with NET map slope
functions.  Although the present work deals primarily with NET maps,
\cite{fpp3} deals almost exclusively with NET maps.  Its goal, not
quite achieved, is to find an effective algorithm which determines
whether a given NET map is equivalent to a rational map.  So its
results are quantitative, whereas the results of the present paper are
more qualitative.

\subsection*{Acknowledgements} The author gratefully acknowledges
support from the American Institute for Mathematics.  Many thanks to
Jim Cannon, Bill Floyd, Russell Lodge and Kevin Pilgrim for many
helpful discussions and comments on earlier versions of this paper.

\section{Slope functions}\label{sec:slopefns}\nosubsections

We define slope functions in this section.

We begin with a Thurston map $f\co S^2\to S^2$ whose postcritical set
$P_f$ has exactly four points.  We wish to mark the pair $(S^2,P_f)$
in order to assign slopes to simple closed curves.  For this, let
$\zG$ be the group of isometries of $\bbR^2$ of the form $x\mapsto
2\zl\pm x$ for $\zl\in \bbZ^2$.  By a marking of $(S^2,P_f)$ we mean
an identification of $S^2$ with $\bbR^2/\zG$ so that $P_f$ is the
image of $\bbZ^2$ in $\bbR^2/\zG$.  This choice of marking is
arbitrary in this section.  Later, when dealing with a NET map given
by a presentation, the marking will be determined by the presentation.
Every homotopy class of simple closed curves in $S^2/P_f$ which are
neither peripheral nor null homotopic contains the image of a line in
$\bbR^2$.  The slope of this line lies in the set
$\overline{\bbQ}=\bbQ\cup \{\infty \}$ of extended rational numbers.
This determines a bijection between the set of these homotopy classes
and $\overline{\bbQ}$.  (See \cite[Proposition 1.5 and Proposition
2.6]{fm}, for example.)  In addition to the slopes in
$\overline{\bbQ}$, we introduce a quantity $\odot $ not in
$\overline{\bbQ}$, called the nonslope.  This symbol is intended to
suggest a loop homotopic to a point.

Now we define the slope function $\zm_f\co \overline{\bbQ}\to
\overline{\bbQ}\cup \{\odot \}$ of $f$.  Let $s\in \overline{\bbQ}$.
Let $\zg$ be a simple closed curve in $S^2\setminus P_f$ with slope
$s$.  If every connected component of $f^{-1}(\zg)$ is either
peripheral or null homotopic, then we set $\zm_f(s)=\odot $.  Suppose
that $\zd$ and $\zd'$ are connected components of $f^{-1}(\zg)$ which
are neither peripheral nor null homotopic.  Then $\zd$ and $\zd'$ are
homotopic in $S^2\setminus P_f$.  Hence they have the same slope $t$.
We set $\zm_f(s)=t$.  This defines the slope function $\zm_f$.

\section{Core arcs}\label{sec:corearcs}\nosubsections

Much can be learned about Thurston maps by studying their pullback
actions on simple closed curves.  More can be learned by also studying
their pullback actions on core arcs.  (See Pilgrim and Tan's
\cite[Theorem 3.2]{pt} for a result on arc pullbacks for general
Thurston maps.)  We introduce core arcs in this section.  Let $f\co
S^2\to S^2$ be a Thurston map whose postcritical set $P_f$ contains
exactly four points.

By a core arc for $(S^2,P_f)$ we mean a closed arc $\za$ in $S^2$
whose endpoints are distinct elements of $P_f$ and these endpoints are
the only points of $\za$ in $P_f$.  Let $\zg$ be a simple closed curve
in $S^2\setminus P_f$ which is neither peripheral nor null homotopic.
The complement of $\zg$ in $S^2$ consists of two open topological
disks, each containing two elements of $P_f$.  So each of these disks
contains a unique homotopy class of core arcs relative to $P_f$.  We
say that such a core arc is a core arc for $\zg$.  Conversely,
starting with a core arc $\za$, the boundary of a small regular
neighborhood of $\za$ is a simple closed curve $\zg$ in $S^2\setminus
P_f$ which is neither peripheral nor null homotopic.  So every core
arc $\za$ is the core arc of some simple closed curve, and there
exists a core arc disjoint from $\za$.  Moreover, if $\za$ is a core
arc for two simple closed curves $\zg$ and $\zg'$, then $\zg$ and
$\zg'$ are homotopic relative to $P_f$.  This allows us to define the
slope of $\za$ to be the slope of $\zg$.  For every slope $s$ there
exist exactly two homotopy classes relative to $P_f$ of core arcs with
slope $s$.

By a lift of a core arc $\za$ we mean a closed arc $\widetilde{\za}$ in
$S^2$ such that $f$ maps $\widetilde{\za}$ homeomorphically onto
$\za$.

Let $f$ be a NET map for the rest of this section.  Let $\za$ and
$\zb$ be disjoint core arcs for $(S^2,P_f)$.  The rest of this section
is devoted to an investigation of $f^{-1}(\za\cup \zb)$.  

For the present discussion, we view $f^{-1}(\za\cup \zb)$ as a graph
$G$.  The edges of $G$ are the $\deg(f)$ lifts of $\za$ together with
the $\deg(f)$ lifts of $\zb$.  The vertices of $G$ are the elements of
the set $f^{-1}(P_f)$.  Since the local degree of $f$ at each of its
critical points is 2, every vertex of $G$ has valence either 1 or 2.
So every connected component of $G$ is either an arc or a simple
closed curve.

Let $A=S^2\setminus (\za\cup \zb)$, an annulus.  Let $B$ be a
connected component of $f^{-1}(A)$.  Then $B$ is an annulus, and the
restriction of $f$ to $B$ is a covering map of some degree $d$.  The
boundary components of $B$ are connected components of $G$.  One of
these boundary components is a connected component of $f^{-1}(\za)$,
and the other is a connected component of $f^{-1}(\zb)$.  If $B$ has a
boundary component which is an arc, then this arc has $d$ edges in
$G$.  If $B$ has a boundary component which is a simple closed curve,
then this simple closed curve has $2d$ edges in $G$.

Suppose that both boundary components of $B$ are arcs.  Then the
closure of $B$ is $S^2$.  So $G$ has only two connected components and
both of them are arcs with $d=\deg(f)$ edges.

Suppose that $B$ has a boundary component which is a simple closed
curve $\zd$.  Then $\zd$ is also a boundary component of a connected
component $B'$ of $f^{-1}(A)$ other than $B$.  Just as $\zd$ has $2d$
edges and the restriction of $f$ to $B$ has degree $d$, the
restriction of $f$ to $B'$ must also have degree $d$.  The union
$B\cup \zd\cup B'$ is an annulus in $S^2$.  If both of its boundary
components are arcs, then the closure of this annulus is $S^2$.
Otherwise there exists a connected component $B''$ of $f^{-1}(A)$
other than $B$ and $B'$ with a boundary component which is a simple
closed curve equal to a boundary component of $B\cup \zd\cup B'$.

Continuing in this way, we obtain the following conclusions.  Exactly
two connected components of $f^{-1}(\za\cup \zb)$ are arcs, each
containing $d$ lifts of either $\za$ or $\zb$ for some positive
integer $d$.  The remaining connected components of $f^{-1}(\za\cup
\zb)$ are simple closed curves containing $2d$ lifts of either $\za$
or $\zb$.  Let $\zg$ be a simple closed curve in $S^2$ with core arcs
$\za$ and $\zb$, so that $\zg$ is a core curve for $A$.  The connected
components of $f^{-1}(\za)$ interlace the connected components of
$f^{-1}(\zb)$, and every connected component of $f^{-1}(\za)$ is
separated from every connected component of $f^{-1}(\zb)$ by a
connected component of $f^{-1}(\zg)$.  We note that this discussion
proves statement 2 of Lemma 1.3 of \cite{cfpp}, namely, that
$f^{-1}(P_f)$ contains exactly four points which are not critical
points.  This discussion also proves part of statement 1 of Theorem
4.1 of \cite{cfpp}, namely, that every connected component of
$f^{-1}(\zg)$ maps to $\zg$ with the same degree.

 Figure~\ref{fig:ovals} illustrates one possibility for
$f^{-1}(\za\cup \zb\cup \zg)$.  The three dashed simple closed curves
form $f^{-1}(\zg)$.  The solid black, respectively gray, curves form
$f^{-1}(\za)$, respectively $f^{-1}(\zb)$.  The dots are the elements
of $f^{-1}(P_f)$.

 \begin{figure}
\centerline{\includegraphics{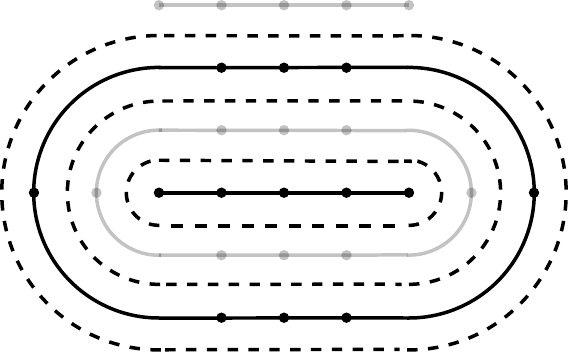}}
\caption{The inverse images of a simple closed curve and two
core arcs}
\label{fig:ovals}
 \end{figure}

\section{Intersection numbers}\label{sec:inttnnums}\nosubsections

In this section we fix notation and conventions concerning the
intersection pairings that will be used later.

We begin by defining intersection numbers of slopes.  Let
$s=\frac{p}{q}$ and $s'=\frac{p'}{q'}$ be two slopes (elements of
$\overline{\bbQ}=\bbQ\cup \{\infty \}$), where $p$, $q$ and $p'$, $q'$
are two pairs of relatively prime integers.  The intersection number
$\zi(s,s')$ is the absolute value of the determinant of the $2\times
2$ matrix $\left[\begin{smallmatrix}p & p' \\ q & q'
\end{smallmatrix}\right]$: $\zi(s,s')=\left|pq'-p'q\right|$.  Viewing
the elements of $\text{PGL}(2,\bbZ)$ as linear fractional
transformations, if $\zv\in \text{PGL}(2,\bbZ)$ and $s$, $s'$ are
slopes, then $\zi(\zv(s),\zv(s'))=\zi(s,s')$.  So $\zi$ is invariant
under the action of $\text{PGL}(2,\bbZ)$.

We return to the situation of a Thurston map $f\co S^2\to S^2$ whose
postcritical set $P_f$ has exactly four points.  Let $\za$ and $\zb$
be either core arcs or simple closed curves in $S^2\setminus P_f$
which are neither peripheral nor null homotopic.  We define the
intersection number $\zi(\za,\zb)$ so that $\zi(\za,\zb)=\zi(s,t)$,
where $s$ and $t$ are the slopes of $\za$ and $\zb$.  We use
bilinearity to extend this definition to the case in which either
$\za$ or $\zb$ is an arc or simple closed curve which is a
concatenation of core arcs.  So in an expression of the form
$\zi(\cdot ,\cdot )$, either both arguments are slopes or else each is
either an arc or simple closed curve.

If $\za$ and $\zb$ are both simple closed curves in $S^2\setminus
P_f$, then $\zi(\za,\zb)$ is the geometric intersection number of two
simple closed curves mapping to $\za$ and $\zb$ in a torus which
double covers $S^2$ with branching exactly over $P_f$.  This is half
the geometric intersection number of $\za$ and $\zb$.  In other words,
points in $\za\cap \zb$ are counted with multiplicity $\frac{1}{2}$.
If one of $\za$ or $\zb$ is a core arc and the other is a simple
closed curve in $S^2\setminus P_f$, then points in $\za\cap \zb$ are
counted with multiplicity 1.  If $\za$ and $\zb$ are distinct core
arcs, then endpoints in $\za\cap \zb$ are counted with multiplicity 1
and all other points are counted with multiplicity 2.

\section{General Lipschitz-type conditions}
\label{sec:genllip}\nosubsections

In this section we work with a general Thurston map with exactly four
postcritical points.  We prove Theorem~\ref{thm:lipschitz}, which
states that slope functions of such maps satisfy a Lipschitz-type
condition with respect to intersection pairing.  We will give two
proofs.  Each proof uses a lemma.

To prepare for the first lemma, let $\bbH$ be the open upper
half-plane in $\bbC$.  We use the discussion in Section 6 of
\cite{cfpp} up to the half-space theorem.  From there we have for a
slope $\frac{p}{q}$ in reduced form and a positive real number $m$
that
  \begin{equation*}
B_m(\tfrac{p}{q})=\{z\in\bbH:\frac{\text{Im}(z)}{\left|pz+q\right|^2}>m\}
  \end{equation*}
is a horoball in $\bbH$ tangent to $\partial \bbH$ at $-\frac{q}{p}$.
According to Section 6 of \cite{cfpp}, the Euclidean diameter of
$B_m(\frac{p}{q})$ is $\frac{1}{mp^2}$.

This leads to the following lemma.

\begin{lemma}\label{lemma:tangent}Suppose that
$s=\frac{p}{q}$ and $s'=\frac{p'}{q'}$ are distinct slopes, where $p$,
$q$ and $p'$, $q'$ are two pairs of relatively prime integers.  Suppose
that $m$, $m'$ are positive real numbers such that the horoballs
$B_m(s)$, $B_{m'}(s')$ are tangent to each other.  Then
  \begin{equation*}
\zi(s,s')=\frac{1}{\sqrt{mm'}}.
  \end{equation*}
\end{lemma}
  \begin{proof} Corollary 6.2 of \cite{cfpp} shows that if $\zv\in
\text{PSL}(2,\bbZ)$ and $t$, $t'$ are slopes such that
$\zv(-1/t)=-1/t'$, then $\zv(B_m(t))=B_m(t')$.  This and the fact that
$\zi$ is invariant under the action of $\text{PSL}(2,\bbZ)$ imply that
to prove the lemma, we may assume that $s'=0$.  This assumption and
the fact that $B_m(s)$ and $B_{m'}(s')$ are tangent to each other
imply that $m'$ is the Euclidean diameter of $B_m(s)$.  Section 6 of
\cite{cfpp} determines this diameter, implying that
$m'=\frac{1}{mp^2}$.  Hence
  \begin{equation*}
\frac{1}{\sqrt{mm'}}=\left|p\right|=\left|p\cdot 1-q\cdot
0\right|=\zi(s,s').
  \end{equation*}
This proves Lemma~\ref{lemma:tangent}.
\end{proof}

Having proved the first lemma, we turn to the second.  For this, let
$P$ be a parabolic element of $\text{SL}(2,\bbZ)$.  Then $P$ acts on
$\partial \bbH$ with exactly one fixed point.  Every other point in
$\partial \bbH$ moves in the same direction, either clockwise or
counterclockwise.  We say that $P$ is positive if every other point
moves in the counterclockwise direction.  For example, the element
$\left[\begin{smallmatrix}1 & 1 \\ 0 & 1
\end{smallmatrix}\right]$ of $\text{SL}(2,\bbZ)$ generates the stabilizer
of $\infty $ in $\text{SL}(2,\bbZ)$ and it is positive.  For every
extended rational number $s$, there exists a positive generator of the
stabilizer of $s$ in $\text{SL}(2,\bbZ)$, and this generator is unique
up to multiplication by $\pm 1$.  This brings us to the second lemma,
where we let tr denote the trace map on $2\times 2$ matrices.

\begin{lemma}\label{lemma:trace} Let $P_1$ and $P_2$ be
parabolic elements of $\text{SL}(2,\bbZ)$.  Suppose that $\pm P_i$ is
the $n_i$th power of a positive generator of the stabilizer of the
extended rational number $s_i$ which $P_i$ fixes for $i\in \{1,2\}$.
Then
  \begin{equation*}
\left|\text{tr}(P_1P_2^{-1})\right|=\left|2+n_1n_2\zi(s_1,s_2)^2\right|.
  \end{equation*}
\end{lemma}
  \begin{proof} The lemma is insensitive to conjugating $P_1$ and
$P_2$ by the same element of $\text{SL}(2,\bbZ)$.  Hence we may assume
that $s_1=\infty $.  The lemma is also insensitive to multiplying
$P_1$ by $-1$.  Hence we may assume that
$P_1=\left[\begin{smallmatrix}1 & n_1 \\ 0 & 1
\end{smallmatrix}\right]$.  We may also assume that
$P_2=Q\left[\begin{smallmatrix}1 & n_2 \\ 0 & 1
\end{smallmatrix}\right]Q^{-1}$ for some matrix
$Q=\left[\begin{smallmatrix}p & r \\ q & s \end{smallmatrix}\right]\in
\text{SL}(2,\bbZ)$.  So $s_2=\frac{p}{q}$.  Then
  \begin{align*}
P_2^{-1} & =Q\left[\begin{matrix}1 & -n_2 \\ 0 & 1
\end{matrix}\right]Q^{-1}
=Q\left(1+\left[\begin{matrix}0 & -n_2 \\ 0 & 0
\end{matrix}\right]\right)Q^{-1}=
1+Q\left[\begin{matrix}0 & -n_2 \\ 0 & 0 \end{matrix}\right]Q^{-1}\\
 & = 1+\left[\begin{matrix}p & r \\ q & s
\end{matrix}\right]\left[\begin{matrix}
0 & -n_2 \\ 0 & 0 \end{matrix}\right]\left[\begin{matrix}s & -r \\ -q &
p \end{matrix}\right]=1+n_2\left[\begin{matrix}0 & -p \\ 0 & -q
\end{matrix}\right]\left[\begin{matrix}s & -r \\ -q & p
\end{matrix}\right]\\
&=1+n_2\left[\begin{matrix}pq & -p^2 \\ q^2 & -pq \end{matrix}\right].
  \end{align*}
So  
  \begin{align*}
\text{tr}(P_1P_2^{-1}) & =
\text{tr}\left(\left(1+n_1\left[\begin{matrix}0 & 1 \\ 0 & 0
\end{matrix}\right]\right)\left(1+n_2\left[\begin{matrix}pq & -p^2 \\
q^2 &
-pq \end{matrix}\right]\right)\right)\\
 & =\text{tr}(1)+\text{tr}\left(n_1\left[\begin{matrix}0 & 1 \\ 0 &
0 \end{matrix}\right]\right)+\text{tr}\left(n_2\left[\begin{matrix}pq &
-p^2 \\ q^2 & -pq \end{matrix}\right]\right)+\text{tr}\left(n_1n_2
\left[\begin{matrix}0 & 1 \\ 0 & 0
\end{matrix}\right]\left[\begin{matrix}pq & -p^2 \\ q^2 & -pq
\end{matrix}\right]\right)\\
&=2+n_1n_2q^2=2+n_1n_2(1\cdot q-0\cdot p)^2=2+n_1n_2\zi(s_1,s_2)^2.
  \end{align*}
This proves Lemma~\ref{lemma:trace}.

\end{proof}

With Lemmas~\ref{lemma:tangent} and \ref{lemma:trace} in hand, all
that we need for Theorem~\ref{thm:lipschitz} is the definition of a
slope multiplier.  For this, let $s\in \overline{\bbQ}$.  Let $\zg$ be
a simple closed curve in $S^2\setminus P_f$ with slope $s$ with respect to a
fixed marking (Section~\ref{sec:slopefns}) of $S^2\setminus P_f$.  Let
$\widetilde{\zg}_1,\dotsc,\widetilde{\zg}_k$ be the connected
components of $f^{-1}(\zg)$ which are neither peripheral nor null
homotopic.  Let $d_i$ be the degree with which $f$ maps
$\widetilde{\zg}_i$ to $\zg$ for $i\in\{1,\ldots,k\}$.  Then the
multiplier of $s$ is $\zr=\sum_{i=1}^{k}d^{-1}_i$.  Because $P_f$
contains only four points, the slopes of
$\widetilde{\zg}_1,\dotsc,\widetilde{\zg}_k$ are equal.  If this slope
equals the slope of $\zg$, then $\zg$ forms an $f$-stable multicurve.
It therefore has a Thurston matrix, which in this case is simply a
$1\times 1$ matrix with entry $\zr$.  The spectral radius of this
matrix is called a Thurston multiplier.  W. Thurston's
characterization of rational maps in this situation implies that if
the orbifold of $f$ is hyperbolic, then $f$ is equivalent to a
rational map if and only if the multiplier of every $f$-stable
multicurve is less than 1.

\begin{thm}\label{thm:lipschitz}Let $f$ be a Thurston map
whose postcritical set $P_f$ has exactly four points.  Let $\zm_f$ be
the slope function of $f$ with respect to a fixed marking of
$S^2\setminus P_f$.  Let $s_1$ and $s_2$ be slopes with multipliers
$\zr_1$ and $\zr_2$ such that $\zm_f(s_1)\ne \odot $ and
$\zm_f(s_2)\ne \odot $.  Then
  \begin{equation*}
\zi(\zm_f(s_1),\zm_f(s_2))\le \frac{1}{\sqrt{\zr_1\zr_2}}\zi(s_1,s_2).
  \end{equation*}
Furthermore, if $s_1\ne s_2$, then the inequality is strict if and
only if the orbifold of $f$ is hyperbolic.
\end{thm}
  \begin{proof} We first prove the inequality using
Lemma~\ref{lemma:tangent}.  Then we give a second proof of the
inequality using Lemma~\ref{lemma:trace}.  The second argument also
proves the statement concerning strictness of the inequality.

Line 6.5 of \cite{cfpp} deals with NET maps, but it holds whenever
$P_f$ contains exactly four points.  It shows that the pullback map
$\zs_f$ of $f$ on the upper half-plane, naturally viewed as the
Teichm\"{u}ller space of $f$, maps $B_m(s_i)$ into
$B_{\zr_im}(\zm_f(s_i))$ for every positive real number $m$ and $i\in
\{1,2\}$.  We choose positive real numbers $m_1$ and $m_2$ so that
$B_{m_1}(s_1)$ and $B_{m_2}(s_2)$ are tangent.  Then $\zs_f$ maps this
point of tangency to the closures of both $B_{\zr_1m_1}(\zm_f(s_1))$
and $B_{\zr_2m_2}(\zm_f(s_2))$.  Because the closures of
$B_{\zr_1m_2}(\zm_f(s_1))$ and $B_{\zr_2m_2}(\zm_f(s_2))$ intersect
nontrivially, Lemma~\ref{lemma:tangent} implies that
  \begin{equation*}
\zi(\zm_f(s_1),\zm_f(s_2))\le
\frac{1}{\sqrt{\zr_1m_1\zr_2m_2}}=
\frac{1}{\sqrt{\zr_1\zr_2}}\frac{1}{\sqrt{m_1m_2}}=
\frac{1}{\sqrt{\zr_1\zr_2}}\zi(s_1,s_2).
  \end{equation*}
This completes our first proof of the inequality of
Theorem~\ref{thm:lipschitz}.

We now give another proof of this inequality using
Lemma~\ref{lemma:trace} instead of Lemma~\ref{lemma:tangent}.  Let $G$
be the pure modular group of $(S^2,P_f)$.  Let $G_f$ be the group of
liftables for $f$ in $G$.  So $\zj\in G_f$ if and only if there exists
$\widetilde{\zj}\in G$ such that $\zj[f]=[f]\widetilde{\zj}$, where
$[f]$ denotes the homotopy class of $f$ relative to $P_f$.  As in
Theorem 7.1 of \cite{cfpp}, we have the following.  Let $i\in
\{1,2\}$.  There exists $\zt_i\in G_f$ whose pullback map
$\zm_{\zt_i}$ on slopes is the $n_i$th power of the positive element
which generates the stabilizer of $s_i$ in $\text{PSL}(2,\bbZ)$ for
some positive integer $n_i$.  Moreover,
$\widetilde{\zt}_i=(\zt'_i)^{\zr_i}$, where $\zt'_i$ is the element of
$G$ such that $\zm_{\zt'_i}$ is the $n_i$th power of the positive
element which generates the stabilizer of $s'_i$ in
$\text{PSL}(2,\bbZ)$.

Now we apply Lemma~\ref{lemma:trace} to matrices $P_1$ and $P_2$ in
$\text{SL}(2,\bbZ)$ which represent $\zm_{\zt_1}$ and $\zm_{\zt_2}$,
respectively, obtaining that
  \begin{equation*}
\left|\text{tr}(P_1P_2^{-1})\right|=2+n_1n_2\zi(s_1,s_2)^2.
  \end{equation*}
We next apply Lemma~\ref{lemma:trace} again, this time to matrices
$P'_1$ and $P'_2$ in $\text{SL}(2,\bbZ)$ which represent
$\zm_{\widetilde{\zt}_1}$ and $\zm_{\widetilde{\zt}_2}$, respectively,
obtaining that
  \begin{equation*}
\left|\text{tr}(P_1'{P'_2}^{-1})\right|=
2+\zr_1n_1\zr_2n_2\zi(\zm_f(s_1),\zm_f(s_2))^2.
  \end{equation*}
There is nothing to prove if $s_1=s_2$, so suppose that $s_1\ne s_2$.
Then $\zi(s_1,s_2)>0$ and $\left|\text{tr}(P_1P_2^{-1})\right|>2$.  So
the matrix $P_1P_2^{-1}$ in $\text{SL}(2,\bbZ)$ which represents
$\zm_{\zt_1}\zm_{\zt_2}^{-1}$ is hyperbolic.  Just as $f$ induces a
pullback map $\zs_f\co \bbH\to \bbH$, we have a pullback map
$\zs_\zj\co \bbH\to \bbH$ for every $\zj\in G$.  Section 6 of
\cite{fkklpps} shows that a fixed conjugation takes $\zm_\zj$ to
$\zs_\zj$ for every $\zj$ in $G$.  Hence $\zs_{\zt_1}\zs_{\zt_2}^{-1}$
is hyperbolic.  Theorem 4.4 of \cite{fpp2} is proven in the context of
NET maps, but it holds whenever $P_f$ has exactly four points.  So
statement 3 of Theorem 4.4 of \cite{fpp2} implies that if the orbifold
of $f$ is hyperbolic, then
  \begin{equation*}
\left|\text{tr}(P'_1{P'_2}^{-1})\right|<\left|\text{tr}(P_1{P_2}^{-1})\right|.
  \end{equation*}
If the orbifold of $f$ is Euclidean, then this inequality is an
equality.  Thus
  \begin{equation*}
\zr_1\zr_2\zi(\zm_f(s_1),\zm_f(s_2))^2\le \zi(s_1,s_2)^2.
  \end{equation*}
This easily gives our second proof of the inequality of
Theorem~\ref{thm:lipschitz} and also the statement about strictness of
the inequality.

\end{proof}

\section{Review of NET map slope function evaluation}
\label{sec:eval}\nosubsections

When considering examples of NET maps later, we will want to evaluate
NET map slope functions.  So in this section we discuss how to evaluate
the slope function $\zm_f$ of a NET map $f$ given by a presentation as
in \cite{fpp1}.  Specifically, we review Theorem 5.1 of \cite{cfpp}
and its visual interpretation.

We begin with a slope $s=\frac{p}{q}$, where $p$ and $q$ are
relatively prime integers.  Theorem 4.1 of \cite{cfpp} implies that if
$\zg$ is a simple closed curve in $S^2\setminus P_f$ with slope $s$,
then every connected component of $f^{-1}(\zg)$ maps to $\zg$ with the
same local degree $d(s)$.  We have a line segment $S$ in $\bbR^2$ with
slope $s$ and endpoints $v$ and $w=v+2d(s)(q,p)$.  We choose $\zg$ and
$S$ so that $S$ maps to one of the connected components of
$f^{-1}(\zg)$ in the usual quotient space $\bbR/\zG_1$.  The line
segment $S$ meets the spin mirrors of $f$ transversely in finitely
many points not equal to either $v$ or $w$.  Let $\zl_1,\dotsc,\zl_n$
be the centers of these spin mirrors in order from $v$ to $w$.  We
have a basis $B$ of the usual lattice $\zL_1$.  The conclusion of
Theorem 5.1 of \cite{cfpp} is that $\zm_f(s)$ is the slope of the line
segment joining $v$ and $w'=(-1)^nw+2\sum_{i=1}^{n}(-1)^{i+1}\zl_i$
with respect to the basis $B$.  Although this conclusion is correct,
it is peculiar because it fails to recognize that $n$ is always even,
and so $(-1)^n=1$.  In fact, this follows from the theorem.  Indeed,
it is implicit in the conclusion that $v-w'\in \zL_1$.  If $n$ is odd,
then $v-w\in \zL_1$, $v+w\in \zL_1$ and hence $2v\in \zL_1$.  But we
may perturb $v$, and so $\zL_1$ contains a line segment in $\bbR^2$.
This is absurd.

In this paragraph we give a perhaps clearer explanation of why the
integer $n$ in the previous paragraph is even.  The line segment $S$
maps to a simple closed curve $\zd$ in $\bbR^2/\zG_1$, and
$f(\zd)=\zg$.  The curve $\zd$ separates two elements of $P_f$ from
the other two elements of $P_f$.  We implicitly have that $f=h\circ
g$, where $g$ is a Euclidean NET map and $h$ is a push map.  As in
Figure~\ref{fig:ovals}, the curve $\zd$ also separates two elements of
$P_g$ from the other two elements of $P_g$.  The image in
$\bbR^2/\zG_1$ of a spin mirror is an arc $\za$ with one endpoint in
$P_g$ and one endpoint in $P_f$.  The curve $\zd$ separates the
endpoints of $\za$ if and only if the intersection number
$\zi(\za,\zd)$ is odd.  Because these four arcs $\za$ pair the
elements of $P_f$ and $P_g$, the sum of the intersection numbers
$\zi(\za,\zd)$ is even.  Thus $S$ meets an even number of spin mirrors
and $n$ is even.

We next discuss by example what might be called the visual or
geometric interpretation of NET map slope function evaluation, as
discussed immediately after Remark 5.2 of \cite{cfpp}.
Figure~\ref{fig:visual1} shows part of $\bbR^2$.  Elements of the
lattice $\bbZ^2$ are marked with dots.  The large dots are elements of
the sublattice $\zL_1$ with basis consisting of $(4,0)$ and $(4,1)$.
We assume that $(4,0)$ and $(4,1)$ are part of a presentation for a
NET map $f$ and that the dashed horizontal line segments in
Figure~\ref{fig:visual1} are spin mirrors determined by this
presentation.  Figure~\ref{fig:visual1} shows a line segment $S$ with
endpoints $v$, $w$ and slope $\frac{1}{3}$.  The points where $S$
meets spin mirrors are labeled $A$, $B$, $C$, $D$.  The image of $S$
in the usual quotient space $\bbR^2/\zG_1$ is a simple closed curve
which $f$ maps to a simple closed curve with slope $\frac{1}{3}$
relative to our presentation.  To compute $\zm_f(\frac{1}{3})$, we
turn to Figure~\ref{fig:visual2}.  We imagine a photon traveling along
the initial segment of $S$ from $v$ to $A'=A$.  At $A'$ the photon
spins about the center of a spin mirror to $A''$.  It then travels
parallel to $S$ to $B'$, where it encounters another spin mirror.
Then it spins to $B''$ and travels to $C'$.  Then it spins to $C''$
and travels to $D'$.  Then it spins to $D''$ and travels to $w'$,
where it stops.  The sum of the lengths of the line segments parallel
to $S$ in Figure~\ref{fig:visual2} is the length of $S$.  The vector
$v-w'$ is in $2\zL_1$.  The slope of the line segment joining $v$ and
$w'$ relative to the basis consisting of $(4,0)$ and $(4,1)$ is
$\frac{2}{-2}=-1$.  So $\zm_f(\frac{1}{3})=-1$.

  \begin{figure}
\centerline{\includegraphics{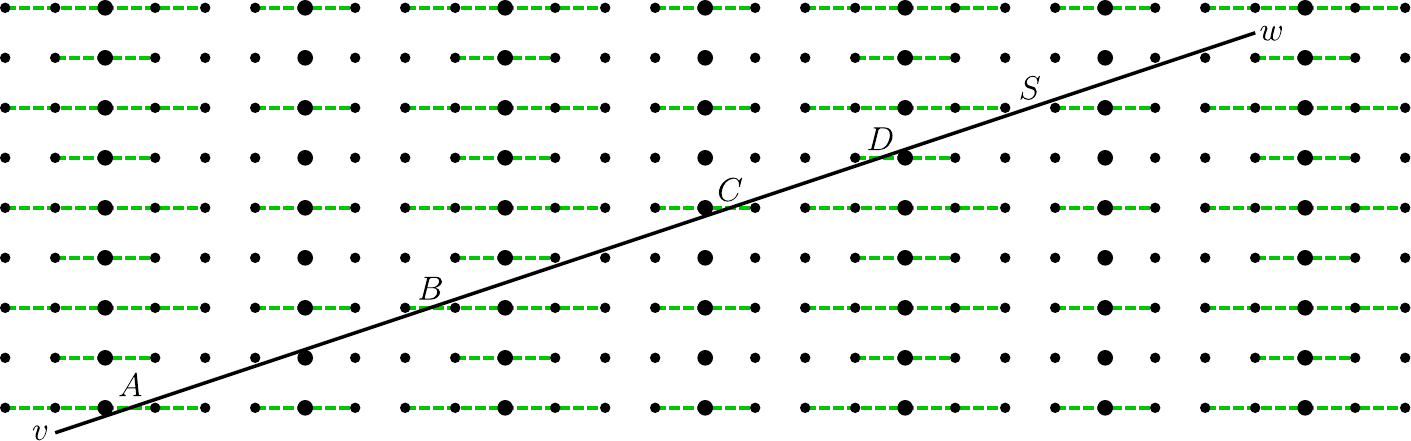}}
\caption{Evaluating a NET map slope function at slope $\frac{1}{3}$}
\label{fig:visual1}
  \end{figure}

  \begin{figure}
\centerline{\includegraphics{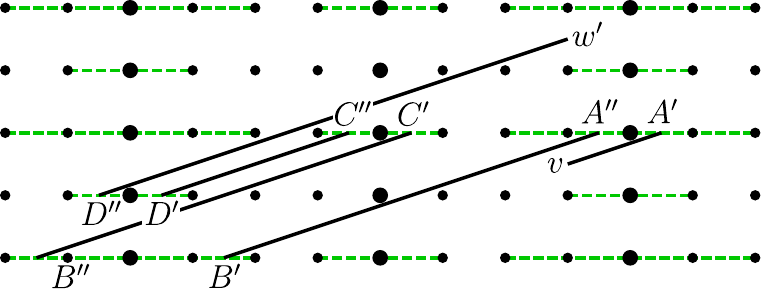}}
\caption{Evaluating a NET map slope function at slope $\frac{1}{3}$}
\label{fig:visual2}
  \end{figure}

\section{Lipschitz-type conditions for NET maps}\label{sec:nmlip}
\nosubsections

Theorem~\ref{thm:lipschitz} provides a Lipschitz-type condition for
the slope function of a Thurston map with exactly four postcritical
points.  This section improves on this in the special case of NET
maps.  Most of the results of this section rest on the following
elementary lemma, which we take for granted.

\begin{lemma}\label{lemma:torus} Let $T^2$ be the standard torus
$\bbR^2/\bbZ^2$ with its natural translation surface structure.  Let
$S^2$ be the 2-sphere half-translation surface gotten from $T^2$ by
identifying every point $x\in T^2$ with $-x$.  Then the following
statements hold.
\begin{enumerate}
  \item Two closed geodesics in $T^2$ with equal slopes have equal lengths.
  \item Let $\zg$ and $\zd$ be closed geodesics in $T^2$ with distinct
slopes $s$ and $t$.  Then the points in $\zg\cap \zd$ partition $\zg$
into $\zi(s,t)$ segments of equal length.
  \item Let $\zg$ and $\zd$ be closed geodesics in $S^2\setminus P$ with
distinct slopes $s$ and $t$, where $P$ is the set of four branch
values of the quotient map from $T^2$ to $S^2$.  Then the points in
$\zg\cap \zd$ partition $\zg$ into $2\zi(s,t)$ segments whose lengths
alternate between two (possibly equal) values.
\end{enumerate}
\end{lemma}

We deal with a NET map $f$ with postcritical set $P_f$ for the rest of
this section.

We continue by introducing some notation.  Let $\zg$ be a simple
closed curve in $S^2\setminus P_f$ which is neither peripheral nor
null homotopic.  It was noted at the end of the penultimate paragraph
of Section~\ref{sec:corearcs} that $f$ maps every connected component
of $f^{-1}(\zg)$ to $\zg$ with the same degree.  Let $d(\zg)$ denote
this number.  Let $c(\zg)$ denote the number of these connected
components which are neither peripheral nor null homotopic.  We write
$c(s)=c(\zg)$ and $d(s)=d(\zg)$, where $s$ is the slope of $\zg$.

Here is the first theorem.  Remark~\ref{remark:nmlip1} shows that it
is a strengthening of Theorem~\ref{thm:lipschitz} in the case of NET
maps.

\begin{thm}\label{thm:nmlip1}Let $\zg$ and $\zd$ be simple
closed curves in $S^2\setminus P_f$ which are neither peripheral nor
null homotopic.  Let $\widetilde{\zg}$ and $\widetilde{\zd}$ be
connected components of $f^{-1}(\zg)$ and $f^{-1}(\zd)$ which are
neither peripheral nor null homotopic.  Then
  \begin{equation*}
\zi(\widetilde{\zg},\widetilde{\zd})\le
\frac{d(\zg)d(\zd)}{\deg(f)}\zi(\zg,\zd).
  \end{equation*}
\end{thm}
  \begin{proof} As in Theorem 2.1 of \cite{cfpp}, we have that
$f=h\circ g$, where $g$ is a Euclidean NET map and $h$ is a
homeomorphism which maps the postcritical set $P_g$ of $g$ to $P_f$.
The map $g$ preserves an affine structure on $S^2$ which is induced by
a half-translation structure as in Lemma~\ref{lemma:torus}.  The
singular points of this half-translation structure are exactly the
postcritical points of $g$.  Let $\zi'$ denote the intersection
pairing of $(S^2,P_g)$.

  We may, and do, choose $\zg$ and $\zd$ so that $h^{-1}(\zg)$ and
$h^{-1}(\zd)$ are geodesics with respect to the half-translation
structure on $(S^2,P_g)$.  Then so are all of the connected components
of $f^{-1}(\zg)=g^{-1}(h^{-1}(\zg))$ and $f^{-1}(\zd)$.  Hence,
letting $|X|$ denote the cardinality of a set $X$,
  \begin{equation*}
\left|\zg\cap \zd\right|=\left|h^{-1}(\zg)\cap h^{-1}(\zd)\right|=
2\zi'(h^{-1}(\zg),h^{-1}(\zd))=2\zi(\zg,\zd).
  \end{equation*}
So
  \begin{equation*}
\left|f^{-1}(\zg)\cap\widetilde{\zd}\right|=
d(\zd)\left|\zg\cap \zd\right|=2d(\zd)\zi(\zg,\zd).
  \end{equation*}

Viewing the connected components of $f^{-1}(\zg)$ and $f^{-1}(\zd)$ as
connected components of $g^{-1}(h^{-1}(\zg))$ and
$g^{-1}(h^{-1}(\zd))$, we see using Lemma~\ref{lemma:torus} that every
connected component of $f^{-1}(\zg)$ meets $\widetilde{\zd}$ in
$|\widetilde{\zg}\cap \widetilde{\zd}|$ points.  Since the connected
components of $f^{-1}(\zg)$ all map to $\zg$ with degree $d(\zg)$, the
number of them is $\deg(f)/d(\zg)$.  Hence
  \begin{equation*}
2\zi(\widetilde{\zg},\widetilde{\zd})\le\left|\widetilde{\zg}\cap
\widetilde{\zd}\right|=\frac{d(\zg)}{\deg(f)}
\left|f^{-1}(\zg)\cap\widetilde{\zd}\right|=
\frac{2d(\zg)d(\zd)}{\deg(f)}\zi(\zg,\zd).
  \end{equation*}

This proves Theorem~\ref{thm:nmlip1}.
 
\end{proof}

\begin{cor}\label{cor:nmlip1}In the situation of
Theorem~\ref{thm:nmlip1},
  \begin{equation*}
\zi(\widetilde{\zg},\widetilde{\zd})\le
\frac{d(\zg)}{c(\zd)}\zi(\zg,\zd).
  \end{equation*}
\end{cor}
  \begin{proof} This follows from Theorem~\ref{thm:nmlip1} and the
fact that $c(\zd)d(\zd)\le \deg(f)$.

\end{proof}

\begin{remark}\label{remark:nmlip1}  It is possible to recover
the inequality in Theorem~\ref{thm:lipschitz} for NET maps from
Corollary~\ref{cor:nmlip1}.  Indeed, Corollary~\ref{cor:nmlip1}
implies that
  \begin{equation*}
\zi(\widetilde{\zg},\widetilde{\zd})\le
\frac{d(\zg)}{c(\zd)}\zi(\zg,\zd) \text{ and }
\zi(\widetilde{\zg},\widetilde{\zd})\le
\frac{d(\zd)}{c(\zg)}\zi(\zg,\zd).
  \end{equation*}
The inequality in Theorem~\ref{thm:lipschitz} for NET maps results
from this and the fact that if two positive real numbers are both
greater than another, then so is their geometric mean.  So
Corollary~\ref{cor:nmlip1} is stronger than Theorem~\ref{thm:lipschitz}
for NET maps.
\end{remark}

We next give an extension of Corollary~\ref{cor:nmlip1} to the case in
which one of the simple closed curves is replaced by a core arc.  This
will be used in the proof of Theorem~\ref{thm:omit1}.

\begin{thm}\label{thm:nmlip2}Let $\za$ be a core arc for
$(S^2,P_f)$, and let $\zd$ be a simple closed curve in $S^2\setminus
P_f$ which is neither peripheral nor null homotopic.  Let
$\widetilde{\za}$ be a union of core arcs in $f^{-1}(\za)$, and let
$\widetilde{\zd}$ be a connected component of $f^{-1}(\zd)$ which is
neither peripheral nor null homotopic.  Let $d(\widetilde{\za})$ be
the number of lifts of $\za$ in $\widetilde{\za}$.  Then
  \begin{equation*}
\zi(\widetilde{\za},\widetilde{\zd})\le
\frac{d(\widetilde{\za})}{c(\zd)}\zi(\za,\zd).
  \end{equation*}
\end{thm}
  \begin{proof} We may, and do, assume that
$\zi(\za,\zd)=\left|\za\cap \zd\right|$.  So
  \begin{equation*}
\left|\widetilde{\za}\cap
f^{-1}(\zd)\right|=d(\widetilde{\za})\zi(\za,\zd).
  \end{equation*}
The set $f^{-1}(\zd)$ contains $c(\zd)$ connected components homotopic
to $\widetilde{\zd}$.  These curves each meet $\widetilde{\za}$ in no
fewer than $\zi(\widetilde{\za},\widetilde{\zd})$ points.  Thus
  \begin{equation*}
c(\zd)\zi(\widetilde{\za},\widetilde{\zd})\le 
d(\widetilde{\za})\zi(\za,\zd).
  \end{equation*}
This proves Theorem~\ref{thm:nmlip2}.

\end{proof}

Just as Theorem~\ref{thm:nmlip2} extends Corollary~\ref{cor:nmlip1},
the next theorem shows that if we replace one of the simple closed
curves in Theorem~\ref{thm:nmlip1} by a core arc, then we obtain
almost the same result.  The previous result is weakened mainly by the
introduction of the ceiling function.  Theorem~\ref{thm:nmlip3} will
be used in the proof of Theorem~\ref{thm:matings}.

\begin{thm}\label{thm:nmlip3}Let $\za$ be a core arc for
$(S^2,P_f)$, and let $\zd$ be a simple closed curve in $S^2\setminus
P_f$ which is neither peripheral nor null homotopic.  Let
$\widetilde{\za}$ be a connected union of core arcs in $f^{-1}(\za)$.
Let $d(\widetilde{\za})$ be the number of lifts of $\za$ in
$\widetilde{\za}$.  Let $\widetilde{\zd}$ be a connected component of
$f^{-1}(\zd)$ which is neither peripheral nor null homotopic.  Then
  \begin{equation*}
\zi(\widetilde{\za},\widetilde{\zd})\le 2\left\lceil
\frac{d(\widetilde{\za})d(\zd)}{2\deg(f)}\zi(\za,\zd)\right\rceil.
  \end{equation*}
\end{thm}
  \begin{proof} Let $\widetilde{\za}'$ be the connected component of
$f^{-1}(\za)$ which contains $\widetilde{\za}$.  Suppose that
$\widetilde{\za}'$ is a simple closed curve.  The case in which
$\widetilde{\za}'$ is an arc will be left to the reader.  Let $\zg$ be
a simple closed curve in $S^2\setminus P_f$ with core arc $\za$.  We
argue as in the proof of Theorem~\ref{thm:nmlip1}, reducing to the
case in which $\widetilde{\za}'$ and $\widetilde{\zd}$ are geodesics
with respect to a half-translation structure.  As in the proof of
Theorem~\ref{thm:nmlip1}, we have that
  \begin{equation*}
\left|\widetilde{\za}'\cap \widetilde{\zd}\right|=
\frac{2d(\zg)d(\zd)}{\deg(f)}\zi(\zg,\zd)=
\frac{2d(\zg)d(\zd)}{\deg(f)}\zi(\za,\zd).
  \end{equation*}
Lemma~\ref{lemma:torus} implies that the points in
$\widetilde{\za}'\cap \widetilde{\zd}$ can be partitioned into two
interlacing subsets $A$ and $B$ such that each partitions
$\widetilde{\za}'$ into segments of equal lengths.  Since
$\widetilde{\za}$ contains $d(\widetilde{\za})$ consecutive lifts of
the $2d(\zg)$ lifts of $\za$ in $\widetilde{\za}'$ and these have
equal lengths,
  \begin{equation*}
\left|\widetilde{\za}\cap A\right|\le \left\lceil
\frac{d(\widetilde{\za})}{2d(\zg)}\left|A\right|\right\rceil
\text{ and }
\left|\widetilde{\za}\cap B\right|\le \left\lceil
\frac{d(\widetilde{\za})}{2d(\zg)}\left|B\right|\right\rceil.
  \end{equation*}
Hence
  \begin{equation*}
\zi(\widetilde{\za},\widetilde{\zd})\le
\left|\widetilde{\za}\cap \widetilde{\zd}\right|=
\left|\widetilde{\za}\cap A\right|+\left|\widetilde{\za}\cap
B\right|\le 
2\left\lceil
\frac{d(\widetilde{\za})d(\zd)}{2\deg(f)}\zi(\za,\zd)\right\rceil .
  \end{equation*}
This yields Theorem~\ref{thm:nmlip3} if $\widetilde{\za}'$ is a simple
closed curve.  Essentially the same argument proves
Theorem~\ref{thm:nmlip3} when $\widetilde{\za}'$ is an arc.

\end{proof}

\section{The half-space theorem}\label{sec:halfsp}\nosubsections

In this section we reprove part of the half-space theorem
\cite[Theorem 6.7]{cfpp}, we strengthen it in the case of NET maps and
we discuss some of its limitations.

We return to the setting of a Thurston map $f$ whose postcritical set
$P_f$ contains exactly four points.  Let $\zm_f$ be the slope function
of $f$ with respect to some marking (Section~\ref{sec:slopefns}) of
$S^2\setminus P_f$.  Let $s$ be a slope.  If $\zm_f(s)\notin \{s,\odot
\}$, then the half-space theorem provides an explicit open interval in
$\partial \bbH$, called an excluded interval, containing $-1/s$ which
does not contain the negative reciprocal of an obstruction for $f$.
Theorem~\ref{thm:halfsp} shows that the following theorem in effect
generalizes this result.

\begin{thm}\label{thm:fixedhalfsp}Let $f$ be a Thurston
map with exactly four postcritical points.  Let $s=\frac{p}{q}$ be a
slope in reduced form, and suppose that $\zm_f(s)=s'=\frac{p'}{q'}$,
also an extended rational number in reduced form.  Let $\zr$ be the
multiplier for $s$.  Let $\zr_0$ be a positive real number.  Then the
interval in $\partial \bbH$ defined by 
  \begin{equation*}
\left|px+q\right|^2<\zr\zr_0\left|p'x+q'\right|^2
  \end{equation*}
does not contain the negative reciprocal of a fixed point of $\zm_f$
with multiplier at least $\zr_0$.
\end{thm}
  \begin{proof} Suppose that $s_2=\frac{p_2}{q_2}$ is a fixed point of
$\zm_f$ with multiplier $\zr_2$.  Theorem~\ref{thm:lipschitz} with
$s_1=s$ implies that
  \begin{equation*}
\sqrt{\zr \zr_2}\zi(\tfrac{p'}{q'},\tfrac{p_2}{q_2})\le
\zi(\tfrac{p}{q},\tfrac{p_2}{q_2}).
  \end{equation*}
Hence
  \begin{equation*}
\zr\zr_2\left|p'q_2-q'p_2\right|^2\le \left|pq_2-qp_2\right|^2
  \end{equation*}
and
  \begin{equation*}
\zr\zr_2\left|p'x+q'\right|^2\le \left|px+q\right|^2,
  \end{equation*}
where $x=-\frac{q_2}{p_2}$.  This easily proves
Theorem~\ref{thm:fixedhalfsp}. 
\end{proof}

\begin{cor}\label{cor:halfsp}Under the assumptions of
Theorem~\ref{thm:fixedhalfsp}, the interval in $\partial \bbH$
defined by
  \begin{equation*}
\left|px+q\right|^2<\zr\left|p'x+q'\right|^2
  \end{equation*}
does not contain the negative reciprocal of the slope of an
obstruction for $f$.
\end{cor}

\begin{cor}\label{cor:fixedpt}Under the assumptions of
Theorem~\ref{thm:fixedhalfsp}, the interval in $\partial \bbH$
defined by
  \begin{equation*}
\deg(f)\left|px+q\right|<\left|p'x+q'\right|
  \end{equation*}
does not contain the negative reciprocal of a fixed point of $\zm_f$.
\end{cor}

We refer to the intervals which appear in
Theorem~\ref{thm:fixedhalfsp}, Corollary~\ref{cor:halfsp} and
Corollary~\ref{cor:fixedpt} as excluded intervals.  The following
theorem shows that the excluded interval in Corollary~\ref{cor:halfsp}
equals the excluded interval in Theorem 6.7 of \cite{cfpp}, thus
providing a somewhat different interpretation of this interval.

\begin{thm}\label{thm:halfsp}If $\frac{p'}{q'}\ne \frac{p}{q}$, then
the excluded interval interval in Corollary~\ref{cor:halfsp} equals
the excluded interval in Theorem 6.7 of \cite{cfpp}.
\end{thm}
  \begin{proof} We have distinct slopes $t=\frac{p}{q}$ and
$t'=\frac{p'}{q'}=\zm_f(t)$, where $p$, $q$ and $p'$, $q'$ are
relatively prime integers.  The interval in Theorem 6.7 of \cite{cfpp}
is described in terms of horoballs $B_m(t)$ and $B_{\zr m}(t')$, where
$\zr$ is the multiplier for $t$.  See Figure~\ref{fig:horoballs}.
These horoballs are tangent.  There is a unique hyperbolic geodesic
tangent to both of these horoballs at their point of tangency.  This
geodesic is the boundary of a hyperbolic half-space $H$ which contains
$B_m(t)$.  The ideal boundary of $H$ is the excluded interval in
Theorem 6.7 of \cite{cfpp}.  

  \begin{figure}
\centerline{\includegraphics{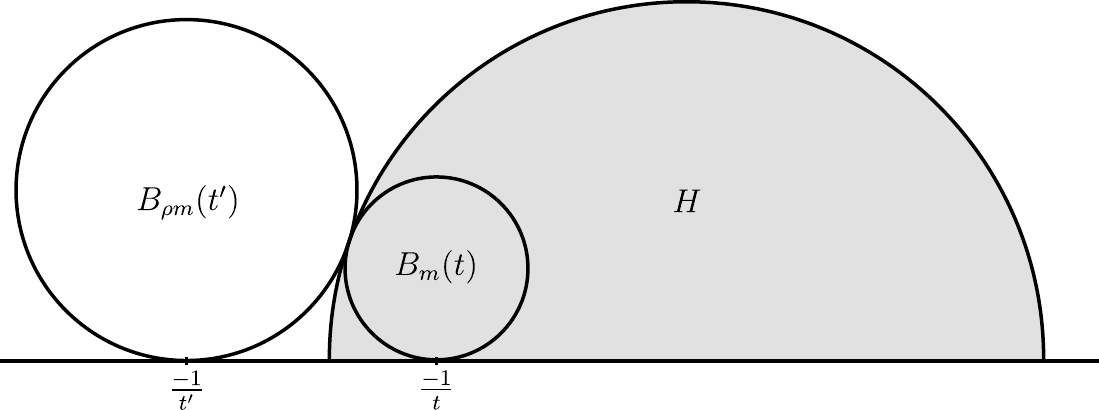}}
\caption{ The horoballs and the
shaded half-space $H$}
\label{fig:horoballs}
  \end{figure}

The statement of Corollary~\ref{cor:halfsp} leads us to consider the
matrix $\left[\begin{smallmatrix}p & q \\ p' & q'
\end{smallmatrix}\right]$.  We multiply its second row by $-1$ if
necessary so that its determinant is positive.  Let
  \begin{equation*}
\zv(z)=\frac{pz+q}{p'z+q'}.
  \end{equation*}
We wish to identify $\zv(H)$.

For this, note that because $\gcd(p',q')=1$, there exist integers $r$
and $s$ such that $\left[\begin{smallmatrix}r & s \\ p' & q'
\end{smallmatrix}\right]\in \text{SL}(2,\bbZ)$.  So there exist
integers $a$, $b$, $c$, $d$ such that
  \begin{equation*}
\left[\begin{matrix}p & q \\ p' & q' \end{matrix}\right]=
\left[\begin{matrix}a & b \\ c & d \end{matrix}\right]
\left[\begin{matrix} r& s \\ p' & q' \end{matrix}\right].
  \end{equation*}
Let
  \begin{equation*}
\zj(z)=\frac{rz+s}{p'z+q'}.
  \end{equation*}
Because $\zv(-\frac{q'}{p'})=\zj(-\frac{q'}{p'})=\infty $, we have
that $c=0$.  Now by computing the second row of the above matrix
product, we find that $d=1$.  By taking determinants, we find that
$a=\zi(t,t')$.  Corollary 6.2 of \cite{cfpp} implies that $\zj(B_{\zr
m}(\frac{p'}{q'}))=B_{\zr m}(0)$.  It follows that $\zv(B_{\zr
m}(\frac{p'}{q'}))=B_{a \zr m}(0)$.  It is clear that
$\zv(B_m(\frac{p}{q}))$ is a horoball tangent to $\partial \bbH$ at 0.
So $\zv(H)$ is the half-space tangent to $\{z\in \bbH:\text{Im}(z)>a
\zr m\}$ at the point $a \zr m\sqrt{-1}$ and whose ideal boundary
contains 0.  Lemma~\ref{lemma:tangent} implies that
$m\sqrt{\zr}=\zi(t,t')^{-1}=a^{-1}$.  So $a \zr m=\sqrt{\zr}$.
Therefore $H$ is the half-space defined by the inequality in
Corollary~\ref{cor:halfsp}.

This proves Theorem~\ref{thm:halfsp}.

\end{proof}

So for Thurston maps with four postcritical points, the excluded
interval in Corollary~\ref{cor:halfsp} is the same as the excluded
interval in the half-space theorem.  Since the proof of
Corollary~\ref{cor:halfsp} is based on the Lipschitz-type inequality
in Theorem~\ref{thm:lipschitz} and Theorem~\ref{thm:nmlip1}
strengthens Theorem~\ref{thm:lipschitz} for NET maps, one might expect
to also obtain a strengthening of Corollary~\ref{cor:halfsp} for NET
maps.  This is what the following theorem does, as can easily be
verified.

\begin{thm}[The NET map half-space theorem]\label{thm:nmhalfsp} Let
$f$ be a NET map.  Let $s=\frac{p}{q}$ be a slope in reduced form, and
suppose that $\zm_f(s)=s'=\frac{p'}{q'}$, also an extended rational
number in reduced form.  Let $d=d(s)$, defined early in
Section~\ref{sec:nmlip}.  Let $e$ be the smallest positive divisor of
$\deg(f)$ such that $e^2\ge \deg(f)$.  Then the interval in $\partial
\bbH$ defined by
  \begin{equation*}
d|px+q|<e|p'x+q'|
  \end{equation*}
does not contain the negative reciprocal of the slope of an
obstruction for $f$.
\end{thm}
  \begin{proof} Let $\zg$ be a simple closed curve in $S^2\setminus
P_f$ with slope $s$, so that $d=d(\zg)$.  Let $\zd$ be a simple closed
curve in $S^2\setminus P_f$ which is an obstruction for $f$.  The
multiplier of $\zd$ is $\frac{c(\zd)}{d(\zd)}$.  Because $\zd$ is an
obstruction, $\frac{c(\zd)}{d(\zd)}\ge 1$.  Because $f$ is a NET map,
Theorem 4.1 of \cite{cfpp} implies that $d(\zd)$ divides $\deg(f)$.
Of course, $c(\zd)d(\zd)\le \deg(f)$.  Thus
  \begin{equation*}
\left(\frac{\deg(f)}{d(\zd)}\right)^2\ge
\frac{c(\zd)d(\zd)\deg(f)}{d(\zd)^2}=\frac{c(\zd)}{d(\zd)}\cdot
\deg(f)\ge \deg(f).
  \end{equation*}
So the minimality of $e$ implies that $\frac{\deg(f)}{d(\zd)}\ge e$.

Now we apply Theorem~\ref{thm:nmlip1} to $\zg$ and $\zd$.  The curve
$\widetilde{\zg}$ in Theorem~\ref{thm:nmlip1} has slope
$\frac{p'}{q'}$, and $\widetilde{\zd}$ is homotopic to $\zd$ relative
to $P_f$.  Theorem~\ref{thm:nmlip1} and the previous paragraph imply
that
  \begin{equation*}
\zi(\widetilde{\zg},\widetilde{\zd})\le
\frac{d(\zg)d(\zd)}{\deg(f)}\zi(\zg,\zd)\le \frac{d}{e}\zi(\zg,\zd).
  \end{equation*}
It is a straightforward matter to complete the proof of
Theorem~\ref{thm:nmhalfsp}.

\end{proof}

In the same way, we have the following NET map analog of
Corollary~\ref{cor:fixedpt}.

\begin{thm}\label{thm:nmfixedpt} Under the assumptions of
Theorem~\ref{thm:nmhalfsp}, the interval in $\partial \bbH$ defined by
  \begin{equation*}
d|px+q|<|p'x+q'|
  \end{equation*}
does not contain the negative reciprocal of a fixed point of $\zm_f$.
\end{thm}

The rest of this section is devoted to investigating limitations of
the half-space theorem.  The main results are in
Theorems~\ref{thm:omit2} and \ref{thm:omit3}.
Theorems~\ref{thm:omit2} and \ref{thm:omit3} rest on
Theorem~\ref{thm:omit1}.  This uses the notation $c(s)$ and $d(s)$
from early in Section~\ref{sec:nmlip}.

\begin{thm}\label{thm:omit1}Let $f$ be a NET map with
postcritical set $P_f$.  Let $\za$ be a core arc for $(S^2,P_f)$
having slope s with respect to a fixed marking of $(S^2,P_f)$.
Suppose that there is a core arc $\widetilde{\za}$ in $f^{-1}(\za)$
with slope s which $f$ maps to $\za$ with degree 1. Then we have the
following.
\begin{enumerate}
  \item If $t$ is a slope with $t\ne s$ which is fixed by $\zm_f$,
then $c(t)=1$.
  \item Every excluded open interval for $f$ arising from the
half-space theorem \cite[Theorem 6.7]{cfpp} omits $-1/s$.
  \item If $c(t)d(t)>1$ for every slope $t$ such that $c(t)\ne0$, then
the closure in $\partial \bbH$ of every excluded interval for $f$
which arises from the half-space theorem \cite[Theorem 6.7]{cfpp}
omits $-1/s$.
\end{enumerate}
\end{thm}

\begin{remark}\label{remark:fixedarc} The assumptions imply that $\za$
has a lift which is a core arc either homotopic relative to $P_f$ to
$\za$ or homotopic relative to $P_f$ to a core arc disjoint from
$\za$.  In both cases, statement 1 follows from the proof of Theorem
3.2 of \cite{pt}.
\end{remark}

  \begin{proof} To prove statement 1, we apply
Theorem~\ref{thm:nmlip2}.  We choose the simple closed curve $\zd$
there to have slope $t$.  The curve $\widetilde{\zd}$ is homotopic to
$\zd$ relative to $P_f$.  We have by assumption that
$d(\widetilde{\za})=1$.  So Theorem~\ref{thm:nmlip2} implies that
$c(t)\zi(s,t)\le \zi(s,t)$.  But $\zi(s,t)\ne 0$ because $t\ne s$.
Hence $c(t)=1$.  This proves statement 1.

To prove statement 2, we apply Theorem~\ref{thm:nmlip2} again.  We
take the core arc $\za$ there to be the present $\za$.  Suppose that
the simple closed curve $\zd$ in Theorem~\ref{thm:nmlip2} has slope
$\frac{p}{q}=t$ and that $\widetilde{\zd}$ has slope
$\frac{p'}{q'}=t'\ne t$, where $p$, $q$ and $p'$, $q'$ are two pairs
of relatively prime integers.  Theorem~\ref{thm:nmlip2} implies that
$c(t)\zi(s,t')\le \zi(s,t)$.  Hence
  \begin{equation*}
c(t)|p'(-1/s)+q'|\le |p(-1/s)+q|,
  \end{equation*}
and so
  \begin{equation*}
|p(-1/s)+q|^2\ge c(t)^2|p'(-1/s)+q'|^2\ge
\frac{c(t)}{d(t)}|p'(-1/s)+q'|^2.
  \end{equation*}
This and Theorem~\ref{thm:halfsp} imply that $-\frac{1}{s}$ is not in
the excluded interval arising from the half space theorem applied to
slope $t$.  This proves statement 2.

To prove statement 3, we focus on the second inequality in the last
display.  If it is an equality, then either $s=t'$ or $c(t)=d(t)=1$.
Since $c(t)d(t)>1$,
  \begin{equation*}
|p(-1/s)+q|^2> \frac{c(t)}{d(t)}|p'(-1/s)+q'|^2.
  \end{equation*}
This proves statement 3.

\end{proof}

\begin{remark}\label{remark:omit1} The condition $c(t)d(t)>1$ in
statement 3 of Theorem~\ref{thm:omit1} is satisfied if both elementary
divisors \cite[Section 8]{fpp1} of $f$ are greater than 1.  Indeed, if
both elementary divisors of $f$ are greater than 1, then $d(t)>1$ for
every slope $t$.  To prove this, let $t=\frac{p}{q}$, expressed in
reduced form.  Statement 1 of Theorem 4.1 in \cite{cfpp} implies that
$d(t)$ equals the order of the image of $(q,p)\in \bbZ^2$ in the group
$\bbZ^2/\zL_1$, with $\zL_1$ \cite[Section 3]{fpp1} as usual.  We have
that $\zL_1\subseteq n\bbZ^2$, where $n$ is the smaller elementary
divisor of $f$.  By assumption, $n>1$.  Because $p$ and $q$ are
relatively prime, $n|d(t)$.  In particular, $d(t)>1$.  Thus
$c(t)d(t)>1$ if $c(t)\ne 0$.
\end{remark}

Now we strengthen the assumptions in Theorem~\ref{thm:omit1} and
obtain a stronger conclusion.

\begin{thm}\label{thm:omit2}Let $f$ be a NET map with postcritical set
$P_f$.  Suppose that neither elementary divisor of $f$ equals 1.  Let
$\za$ be a core arc for $(S^2,P_f)$ having slope s with respect to a
fixed marking of $(S^2,P_f)$.  Suppose that there is a core arc in
$f^{-1}(\za)$ with slope s which $f$ maps to $\za$ with degree 1.
Suppose that $s$ is not the slope of an obstruction.  Then $f$ is
Thurston equivalent to a rational map.  Moreover, the closure of every
excluded interval for $f$ which arises from the half-space theorem
\cite[Theorem 6.7]{cfpp} omits $-1/s$.  Hence every finite union of
these excluded intervals omits an interval of $\partial \bbH$.
\end{thm}
  \begin{proof} We begin by proving that the orbifold of $f$ is
hyperbolic.  We proceed by contradiction.  Suppose that the orbifold
of $f$ is Euclidean.  In this case, the assumption that
$\widetilde{\za}$ maps to $\za$ with degree 1 implies that $d(s)=1$.
This contradicts Remark~\ref{remark:omit1}.  Thus the orbifold of $f$
is hyperbolic.

So to prove that $f$ is rational, it suffices to prove that it has no
obstruction.  We have assumed that $s$ is not the slope of an
obstruction.  Statement 1 of Theorem~\ref{thm:omit1} implies that if
$t$ is a slope with $t\ne s$ which is fixed by $\mu_f$, then $c(t)=1$.
Remark~\ref{remark:omit1} shows that $d(t)>1$ because neither
elementary divisor of $f$ is 1.  So $c(t)/d(t)<1$.  So $t$ is not an
obstruction, and $f$ is Thurston equivalent to a rational map.

The rest of Theorem~\ref{thm:omit2} follows from statement 3 of
Theorem~\ref{thm:omit1}. 

\end{proof}

\begin{remark}\label{remark:omit2}It is easy to satisfy the conditions
of Theorem~\ref{thm:omit2}.  For example, the NET map with the
presentation diagram (for which see \cite{fpp1}) in
Figure~\ref{fig:omit2} satisfies the conditions of
Theorem~\ref{thm:omit2}.  We may take $s$ to be either 0 or $\infty $.
Moreover, one can verify using Lemma 4.2 of \cite{cfpp} that every
multiplier is less than 1 for this example, and so the modular group
Hurwitz class (for which see \cite{fpp2}) of this map is completely
unobstructed.
\end{remark}

  \begin{figure} \centerline{\includegraphics{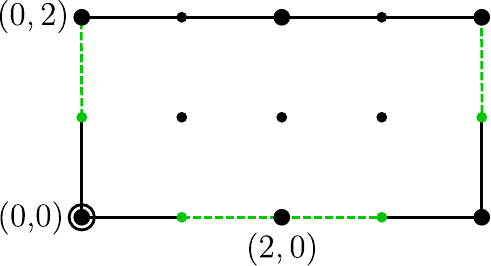}}
\caption{ A presentation diagram for a NET map which satisfies
the conditions of Theorem~\ref{thm:omit2}}
\label{fig:omit2}
  \end{figure}

The next theorem provides conditions under which the union of the
excluded intervals arising from the half-space theorem omits
infinitely many points.

\begin{thm}\label{thm:omit3}Let $f$ be a NET map with
postcritical set $P_f$.  Let $\za$ be a core arc for $(S^2,P_f)$
having slope s with respect to a fixed marking of $(S^2,P_f)$.
Suppose that there is a core arc in $f^{-1}(\za)$ with slope s which
$f$ maps to $\za$ with degree 1.  Suppose that $t\ne s$ is the slope
of an obstruction.  Then there exists an infinite sequence of points
in $\partial \bbH$ converging to $-1/t$ such that no element in this sequence
is contained in an excluded interval for $f$ which arises from the
half-space theorem \cite[Theorem 6.7]{cfpp}.
\end{thm}
  \begin{proof} Statement 1 of Theorem~\ref{thm:omit1} implies that
$c(t)=1$.  Since $t$ is an obstruction, $c(t)/d(t)\ge 1$.  So $d(t)=1$
and the multiplier of $t$ is 1.  

Now let $\zt$ be a (nonzero power of a primitive) Dehn twist about a
simple closed curve with slope $t$ whose homotopy class in the modular
group of $f$ is liftable.  The fact that $t$ is an obstruction for $f$
with multiplier 1 implies that $f$ and $\zt$ commute up to homotopy
relative to $P_f$.  It follows that some core arc in
$f^{-1}(\zt^n(\za))$ with the same slope as $\zt^n(\za)$ maps to
$\zt^n(\za)$ with degree 1 for every integer $n$.  Let $\zs_\zt$
denote the pullback map induced by $\zt$ on $\bbH$, much as $f$
induces a pullback map $\zs_f$ on $\bbH$.  Now statement 2 of
Theorem~\ref{thm:omit1} implies that every excluded interval for $f$
which arises from the half-space theorem omits $\zs_\zt^n(-1/s)$ for
every integer $n$.  These points converge to the fixed point of
$\zs_\zt$, namely $-1/t$, as $n$ tends to $\infty $.

This proves Theorem~\ref{thm:omit3}.

\end{proof}

\begin{remark}\label{remark:omit3}It is easy to satisfy the conditions
of Theorem~\ref{thm:omit3}.  For example, the NET map with the
presentation diagram (for which see \cite{fpp1}) in
Figure~\ref{fig:omit3} satisfies the conditions of
Theorem~\ref{thm:omit3} with $s=0$ and $t=\infty $. This is 21HClass3
in the census of modular group Hurwitz class representatives in
\cite{NET}.  The proof of Theorem~\ref{thm:omit3} implies that no
excluded interval for this NET map contains the reciprocal of an
integer.  In fact, the output for the computer program {\tt NETmap} in
\cite{NET} suggests that no excluded interval for this NET map
contains 2 times the reciprocal of an integer.
\end{remark}

  \begin{figure} \centerline{\includegraphics{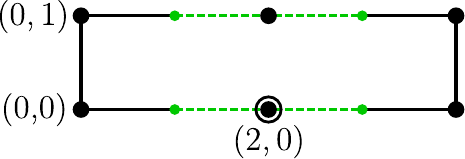}}
\caption{A presentation diagram for a NET map which satisfies
the conditions of Theorem~\ref{thm:omit3}}
\label{fig:omit3}
  \end{figure}

\section{A property of the modular group virtual multi-endomorphism
}\label{sec:fixedsgp}\nosubsections

The goal of this section is to prove a somewhat technical result in
Lemma~\ref{lemma:fixedsgp}.  This will be used in
Section~\ref{sec:bdfixed} to bound the number of slope function fixed
points.  In this section and the next, we work with a general Thurston
map with four postcritical points, not necessarily a NET map.  We need
some preparations to state Lemma~\ref{lemma:fixedsgp}.

Let $f$ be a Thurston map with exactly four postcritical points.  Let
$G$ be the modular group of $f$.  It is well known \cite[Proposition
2.7]{fm} that $G\cong \text{PSL}(2,\bbZ)\ltimes (\bbZ/2\bbZ)^2$.
Following Section 3 of \cite{fpp2}, we say that an element of $G$ is
elliptic, parabolic or hyperbolic according to whether its
$\text{PSL}(2,\bbZ)$-factor is elliptic, parabolic or hyperbolic.  The
remaining elements of $G$ are translations in $T=(\bbZ/2\bbZ)^2$.
Suppose that we have an element $\zj\in G$ whose first factor is the
image of $\left[\begin{smallmatrix}a & b \\ c & d
\end{smallmatrix}\right]\in \text{SL}(2,\bbZ)$.  Then Proposition 4.1
of \cite{fpp2} shows that the pullback map $\zs_\zj\co \bbH\to \bbH$
induced by $\zj$ is given by $\zs_\zj(z)=\frac{dz+b}{cz+a}$.  So, as
in Corollary 4.2 of \cite{fpp2}, the linear fractional transformation
$\zs_\zj$ is elliptic, parabolic or hyperbolic if and only if $\zj$ is
elliptic, parabolic or hyperbolic.

An element $\zj\in G$ is liftable for $f$ if and only if there exists
$\widetilde{\zj}\in G$ such that $\zj[f]=[f]\widetilde{\zj}$, where
$[f]$ is the homotopy class of $f$ relative to its postcritical set.
The element $\widetilde{\zj}$ is a lift of $\zj$.  The set $G_f$ of
all liftables for $f$ in $G$ is a subgroup of $G$.  The element
$\widetilde{\zj}$ need not be unique, and so the assignment
$\zj\mapsto \widetilde{\zj}$ is a multifunction, the modular group
virtual multi-endomorphism for $f$.  This multi-endomorphism maps the
identity element of $G_f$ to a subgroup $\text{DeckMod}(f)$ of $G$.
For every element of $G_f$, the set of all its lifts is a right coset
of $\text{DeckMod}(f)$.

Now we assume that the Thurston pullback map $\zs_f$ is not constant.
If $\zj\in \text{DeckMod}(f)$, then $[f]\zj=[f]$.  Hence $\zs_\zj\circ
\zs_f=\zs_f$.  Thus $\zs_\zj$ is the identity element because $\zs_f$
is not constant.  It follows that the map $\zj\mapsto
\zs_{\widetilde{\zj}}$ for $\zj\in G_f$ is a group antihomomorphism:
$\zs_{\widetilde{\zv \zj}}=\zs_{\widetilde{\zj}}\zs_{\widetilde{\zv}}$
for $\zv,\zj\in G_f$.  In particular, it is a function even if the
modular group virtual multi-endomorphism is not.  The same argument
shows that this multi-endomorphism maps translations in $G_f$ to
translations.

This brings us to Lemma~\ref{lemma:fixedsgp}.

\begin{lemma}\label{lemma:fixedsgp}Let $f$ be a Thurston
map with exactly four postcritical points.  We assume that the
orbifold of $f$ is hyperbolic and that the Thurston pullback map
$\zs_f$ of $f$ is not constant.  Let $\zv$ be an element of
$\text{PSL}(2,\bbZ)$.  Then the set $K$ of all elements $\zj\in G_f$
such that $\zs_{\widetilde{\zj}}=\zv^{-1}\zs_\zj \zv$ is a subgroup of
$G_f$.  Furthermore, $T\cap G_f$ is a normal subgroup of $K$,
$K/(T\cap G_f)$ is cyclic and $K$ contains no hyperbolic elements.
\end{lemma}
  \begin{proof} Using the fact that the map $\zj\mapsto
\zs_{\widetilde{\zj}}$ for $\zj\in G_f$ is a group antihomomorphism,
one verifies that $K$ is a subgroup of $G_f$.  Because the modular
group virtual multi-endomorphism of $f$ maps translations in $G_f$ to
translations, $T\cap G_f\subseteq K$.  This is a normal subgroup of
$G_f$ as well as $K$ because $T$ is a normal subgroup of $G$.
Although stated for NET maps, Theorem 4.4 of \cite{fpp2} holds in the
present more general situation.  Statement 3 of Theorem 4.4 of
\cite{fpp2} implies that if $\zs_\zj$ is hyperbolic, then the absolute
value of the trace of $\zs_{\widetilde{\zj}}$ is less than the
absolute value of the trace of $\zs_\zj$.  Because conjugation
preserves trace, it follows that $K$ contains no hyperbolic
elements. Thus to prove Lemma~\ref{lemma:fixedsgp}, it only remains to
prove that $K/(T\cap G_f)$ is cyclic.

To prove that $K/(T\cap G_f)$ is cyclic, it suffices to prove that
every subgroup of $\text{PSL}(2,\bbZ)$ which is not cyclic contains a
hyperbolic element.  So let $H$ be a subgroup of $\text{PSL}(2,\bbZ)$
which is not cyclic.  We use the fact that $\text{PSL}(2,\bbZ)$ is the
free product of $\bbZ/2\bbZ$ and $\bbZ/3\bbZ$.  (See
\cite[I.4.2]{serre}.)  The corollary in \cite[I.4.3]{serre} implies
that if $H$ is finite, then it is conjugate to a subgroup of one of
these free factors.  Since $H$ is not cyclic, it must be infinite.
Hence it has a nontrivial intersection with the kernel of a surjective
group homomorphism from $\text{PSL}(2,\bbZ)$ to $(\bbZ/2\bbZ)\times
(\bbZ/3\bbZ)$ arising from the free product decomposition of
$\text{PSL}(2,\bbZ)$.  Proposition 18 of \cite[I.4.3]{serre} implies
that this kernel is a free group.  Hence $H$ contains an element of
infinite order.  Since $H$ contains an element of infinite order, it
contains either a hyperbolic element or a parabolic element.  In the
former case we are done.  So suppose that $H$ contains a parabolic
element.  Then $H$ contains a positive parabolic element $\zh_1$ in
the sense of Lemma~\ref{lemma:trace}.  Suppose that $\zh_1$ fixes the
cusp $s_1$.  Since H is not cyclic, it contains an element $\zg$ such
that $\zg(s_1)=s_2\ne s_1$.  Then $\zh_2=\zg \zh_1\zg^{-1}$ is a
positive parabolic element of $H$ which fixes $s_2$.
Lemma~\ref{lemma:trace} now implies that the trace of a matrix in
$\text{SL}(2,\bbZ)$ which represents $\zh_1\zh_2^{-1}$ is greater than
2.  So $\zh_1\zh_2^{-1}$ is a hyperbolic element of $H$.  This proves
that every subgroup of $\text{PSL}(2,\mathbb{Z})$ which is not cyclic
contains a hyperbolic element.

\end{proof}

\section{Bounding the number of fixed points
}\label{sec:bdfixed}\nosubsections

This section is devoted to bounding the number of slope function fixed
points of a Thurston map with four postcritical points.  The statement
of the result in Theorem~\ref{thm:fixedpts} requires some definitions.

We prepare for these definitions by discussing two modular group
actions on $\overline{\bbQ}=\bbQ\cup \{\infty \}$.  Let $f$ be a
Thurston map with postcritical set $P_f$ containing exactly four
points.  Let $\zm_f\co \overline{\bbQ}\to \overline{\bbQ}\cup
\{\odot\}$ be its slope function and let $\zs_f\co \bbH\to \bbH$ be
its Thurston pullback map with respect to some marking
(Section~\ref{sec:slopefns}) of $(S^2,P_f)$.  Selinger proved in
\cite{S} that $\zs_f$ extends continuously to augmented
Teichm\"{u}ller space $\bbH\cup \overline{\bbQ}$.  If $s\in
\overline{\bbQ}$ and $\zm_f(s)\in \overline{\bbQ}$, then
$\zm_f(s)=-1/\zs_f(-1/s)$.  Let $G$ be the modular group of
$(S^2,P_f)$.  Every element $\zj\in G$ induces a pullback map
$\zm_\zj\co \overline{\bbQ}\to \overline{\bbQ}$ on slopes and a
pullback map $\zs_\zj\co \bbH\cup \overline{\bbQ}\to \bbH\cup
\overline{\bbQ}$ on augmented Teichm\"{u}ller space.  The map
$z\mapsto -1/z$ conjugates these linear fractional transformations to
each other.  So, we obtain a slope (right) action of $G$ on
$\overline{\bbQ}$ by means of the maps $\zm_\zj$ and we obtain a cusp
(right) action of $G$ on $\overline{\bbQ}$ by means of the maps
$\zs_\zj$.  The map $z\mapsto -1/z$ conjugates one action to the
other.

Now we make the aforementioned definitions.  As in the previous
section, let $G_f$ be the group of liftables for $f$ in the modular
group $G$.  We define three positive integers $C_f$, $D_f$ and $E_f$.
Proposition 3.1 of \cite{kps} implies that the index of $G_f$ in $G$
is finite.  Because $G$ acts transitively on $\overline{\bbQ}$ and
$[G:G_f]<\infty $, the group $G_f$ acts on $\overline{\bbQ}$ via
either its action on slopes or its action on cusps with the same
finite number of orbits.  Let $C_f$ be the number of these orbits.
Let $D_f$ be 2 times the least common multiple of the positive
integers less than or equal to $\deg(f)$.  If $G_f$ contains an
elliptic element of order 3, then $E_f=3$.  Otherwise, if $G_f$
contains an elliptic element of order 2, then $E_f=2$.  Otherwise,
$E_f=1$.  This notation is chosen because $C_f$ deals with
$G_f$-orbits of cusps, $D_f$ deals with the degree of $f$ and $E_f$
deals with the elliptic elements of $G_f$.  

When referring to slope function fixed points, we are referring to
elements of $\overline{\bbQ}$, not the nonslope.  This brings us to
Theorem~\ref{thm:fixedpts}.

\begin{thm}\label{thm:fixedpts}Let $f$ be a Thurston map
with exactly four postcritical points whose orbifold is hyperbolic.
Let $\zm_f$ be its slope function, and let $G_f$ be the subgroup of
liftables in its modular group $G$.
\begin{enumerate}
  \item If $\zm_f$ has more than $C_fD_fE_f$ fixed points, then $f$
has an obstruction with multiplier 1.
  \item If $f$ has an obstruction with multiplier 1 and $\zm_f$ has
more than one fixed point, then $\zm_f$ has infinitely many fixed
points.
  \item If $f$ has an obstruction with multiplier 1, then the
stabilizer in $G_f$ of this obstruction acts on the set of fixed
points of $\zm_f$ with at most $C_fD_f$ orbits.
\end{enumerate}
\end{thm}

\begin{remark}\label{remark:multiplier1} We note that $f$ has an
obstruction with multiplier 1 if and only if there exists a Dehn twist
in the modular group of $f$ which commutes with $f$ up to homotopy
relative to the postcritical set of $f$.
\end{remark}

  \begin{proof} Rather than working directly with slopes, we work with
cusps, that is, negative reciprocals of slopes in the boundary of the
upper half-plane $\bbH$.  Let $\zs_f$ be the Thurston pullback map of
$f$ on $\bbH$.  Similarly, we let $\zs_\zj$ denote the pullback map of
$\zj\in G$.  We use the action of $G$ on $\overline{\bbQ}$ which
arises from the group antihomomorphism $\zj\mapsto \zs_{\zj}$.  As
discussed before Lemma~\ref{lemma:fixedsgp} in the previous section,
the assignment $\zj\mapsto \zs_{\widetilde{\zj}}$ is a function, a
group antihomomorphism, even when the modular group virtual
multi-endomorphism $\zj\mapsto \widetilde{\zj}$ is not a function.

To prove statement 1, suppose that $\zm_f$ has more than $C_fD_fE_f$
fixed points.  Then $\zs_f$ fixes more than $C_fD_fE_f$ cusps in
$\overline{\bbQ}$.  Since the cusp action of $G_f$ on
$\overline{\bbQ}$ has $C_f$ orbits, one of these orbits contains more
than $D_fE_f$ cusps $t_1,t_2,t_3,\dotsc,t_N$ fixed by $\zs_f$.  So $N>
D_fE_f$.  Since all obstructions are homotopic, we may, and do, choose
$t_1$ so that its negative reciprocal is not the slope of an
obstruction.

Since $t_1,\dotsc,t_N $ are in the same orbit for the action of $G_f$
on $\overline{\bbQ}$, there exist elements $\zj_1,\dotsc,\zj_N$ in
$G_f$ such that $\zs_{\zj_n}(t_1)=t_n$ for every $n\in\{1,\ldots,N\}$.
Let $n\in\{1,\ldots,N\}$.  Then
  \begin{equation*}
t_n=\zs_f(t_n)=\zs_f(\zs_{\zj_n}(t_1))=
\zs_{\widetilde{\zj}_n}(\zs_f(t_1))=\zs_{\widetilde{\zj}_n}(t_1).
  \end{equation*}
So $\zs_{\zj_n}^{-1}\zs_{\widetilde{\zj}_n}$ fixes $t_1$.  The image
in $\text{PSL}(2,\bbZ)$ of the stabilizer of $t_1$ in $G$ is a cyclic
(maximal parabolic) subgroup of $\text{PSL}(2,\bbZ)$.  Let $\zt$ be
the square root of a Dehn twist in $G$ such that $\zs_\zt$ generates
this cyclic subgroup.  Then
$\zs_{\zj_n}^{-1}\zs_{\widetilde{\zj}_n}=\zs_{\zt}^k$ for some integer
$k$.

In this paragraph we normalize $\zj_n$ to bound the absolute value of
$k$ independently of $n$.  Let $P_f$ be the postcritical set of $f$.
Let $\zg$ be a simple closed curve in $S^2\setminus P_f$ with slope
$t_1$.  Consider the connected components of $f^{-1}(\zg)$.  Let $d$
be the least common multiple of the degrees with which $f$ maps them
to $\zg$.  The multiplier $\zr$ of $t_1$ is the sum of the reciprocals
of these degrees for which the associated connected component of
$f^{-1}(\zg)$ is neither peripheral nor null homotopic.  So $\zr
d=c\in \bbZ$.  Then, as in Theorem 7.1 of \cite{cfpp}, $\zt^{2d}\in
G_f$ and $\widetilde{\zt}^{2d}=\zt^{2c}$.  We maintain
$n\in\{1,\ldots,N\}$ and the integer $k$ in the previous paragraph.
Let $\zv_n=\zt^{2dr}\zj_n$ for an integer $r$, soon to be made
precise.  Then $\zv_n\in G_f$,
  \begin{equation*}
\zs_{\zv_n}(t_1)=\zs_{\zj_n}(\zs_{\zt^{2dr}}(t_1))=\zs_{\zj_n}(t_1)=t_n
  \end{equation*}
and
  \begin{equation*}
\zs_{\zv_n}^{-1}\zs_{\widetilde{\zv}_n}=
\zs_{\zt^{2dr}}^{-1}\zs_{\zj_n}^{-1}
\zs_{\widetilde{\zj}_n}\zs_{\zt^{2cr}}=
\zs_{\zt^{2dr}}^{-1}\zs_{\zt}^k\zs_{\zt^{2cr}}=\zs_{\zt}^{k+2r(c-d)}.
  \end{equation*}
Because $t_1$ is not the negative reciprocal of an obstruction, $d>c$.
In particular, $c-d\ne 0$.  This shows that by replacing $\zj_n$ by
$\zt^{2dr}\zj_n$ for the appropriate choice of $r$, we have that $0\le
k<2(d-c)$.  We do this so that this inequality holds for every
$n\in\{1,\ldots,N\}$.

Clearly $2(d-c)<2d\le D_f$.  We also have that $N> D_fE_f$.  So there
exists an integer $k$ and a subset $\cN$ of $\{1,\dotsc,N\}$ with
$\left|\cN\right|>E_f$ such that
$\zs_{\zj_n}^{-1}\zs_{\widetilde{\zj}_n}=\zs_{\zt}^k$ for every $n\in
\cN$.

Now let $m,n\in \cN$.  Let $\zj=\zj_m^{-1} \zj_n$.  Then
  \begin{equation*}
\zs_{\widetilde{\zj}}=\zs_{\widetilde{\zj}_n}\zs_{\widetilde{\zj}_m}^{-1}=
\zs_{\zj_n}\zs_{\zt}^k \zs_{\zt}^{-k}\zs_{\zj_m}^{-1}=
\zs_{\zj_n}\zs_{\zj_m}^{-1}=\zs_\zj.
  \end{equation*}
It follows that $\zj_m^{-1}\zj_n$ belongs to the set $K$ of all
elements $\zj\in G_f$ such that $\zs_{\widetilde{\zj}}=\zs_\zj$ for
every $m,n\in \cN$.  Lemma~\ref{lemma:fixedsgp} with $\zv=1$ implies
that $K$ is a subgroup of $G$ and that the image of $K$ in
$\text{PSL}(2,\bbZ)$ is a cyclic group which contains no hyperbolic
elements.  Since $\left|\cN\right|>E_f$, this image is infinite,
generated by a parabolic element.  Let $\zf$ be an element of $K$
which maps to this parabolic element in the image of $K$.  Then $\zf$
fixes the homotopy class of a simple closed curve $\zg$ in
$S^2\setminus P_f$.  Because $\zs_{\widetilde{\zf}}=\zs_\zf$, the
curve $\zg$ is $f$-stable and its multiplier is 1.  Thus $\zg$ is an
obstruction with multiplier 1.  This proves statement 1 of
Theorem~\ref{thm:fixedpts}.

Minor modifications of the above argument prove statement 3.

We turn our attention to statement 2.  Suppose that $f$ has an
obstruction with multiplier 1 and more than one slope function fixed
point.  Let $t_0$ be the negative reciprocal of the obstruction's
slope.  Then $\zs_f$ fixes $t_0$ and one other extended rational
number $t$.  The group $G_f$ contains the homotopy class $\zt$ of a
(nonzero power of a primitive) Dehn twist about the obstruction.
Because the obstruction's multiplier is 1, this homotopy class is
fixed by the modular group virtual multi-endomorphism $\zj\mapsto
\widetilde{\zj}$ of $f$.  So
  \begin{equation*}
\zs_f\zs_\zt=\zs_{\widetilde{\zt}}\zs_f=\zs_\zt \zs_f.
  \end{equation*}
Hence for every integer $n$ we have that 
  \begin{equation*}
\zs_f(\zs_\zt^n(t))=\zs_\zt^n(\zs_f(t))=\zs_\zt^n(t).
  \end{equation*}
Thus $\zs_f$ fixes every element in the orbit of $t$ under the action
of $\left<\zs_\zt\right>$ on the set of extended rational numbers.
This implies that $\zm_f$ has infinitely many fixed points.  

This completes the proof of Theorem~\ref{thm:fixedpts}.

\end{proof}

\section{NET maps with many formal matings}
\label{sec:matings}\nosubsections

Using \cite{fpp2}, Bill Floyd enumerated all possible NET map dynamic
portraits \cite[Section 9]{fpp2} through degree 40 and ran the
computer program {\tt NETmap} on NET maps having these dynamic
portraits.  The results are in the NET map website \cite{NET}.  These
computations revealed that some NET maps have many formal matings.
The examples presented in this section are based on these
computations.

To describe these NET maps, let $n$ be an integer with $n\ge 4$.  Let
$f_n$ be the NET map with the following presentation (for which see
\cite{fpp1}).  The associated affine map $\zF(x)=Ax+b$ has matrix
$A=\left[\begin{smallmatrix}n & -1 \\ 0 & 1\end{smallmatrix}\right]$
and translation term $b=(n,0)$.  There is a green line segment joining
$(0,0)$ and $(1,0)$.  There is a green line segment joining $(n,0)$
and $(2,0)$.  The green line segments containing $(-1,1)$, $(n-1,1)$
and $(2n-1,1)$ are trivial.  This defines a NET map $f_n$ with degree
$n$ up to Thurston equivalence.  The presentation diagram for $f_5$ is
shown in Figure~\ref{fig:prendgm}.

  \begin{figure}
\centerline{\includegraphics{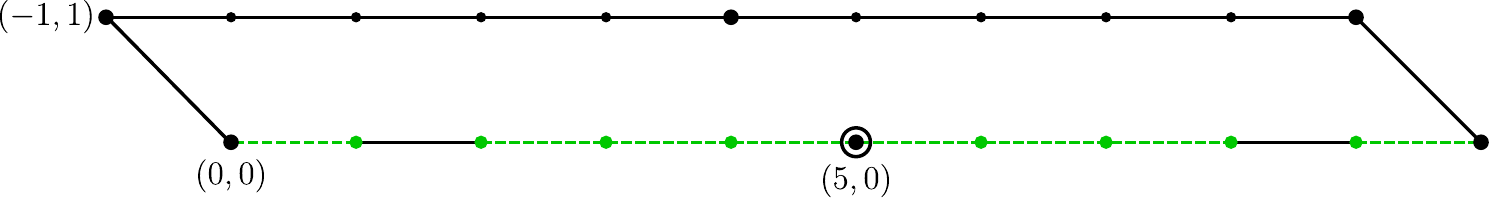}}
\caption{The presentation diagram for $f_5$}
\label{fig:prendgm}
  \end{figure}

\begin{thm}\label{thm:matings}Let $n$ be an integer with
$n\ge 4$.  Then the map $f_n$ is Thurston equivalent to a rational map
and it can be expressed as a formal mating in at least $\left\lceil
\frac{n}{2}\right\rceil$ ways.  The slopes of $\left\lceil
\frac{n}{2}\right\rceil$ mating equators have the form
$\frac{2m}{n-2m-1}$, where $m$ is an integer with $0\le m\le
\left\lceil \frac{n-2}{2}\right\rceil $.
\end{thm}
  \begin{proof} To prove that $f_n$ is Thurston equivalent to a
rational map, we apply W. Thurston's characterization theorem.
The map $f_n$ has a hyperbolic orbifold.  So it suffices to prove that
$f_n$ has no obstruction.

We proceed by contradiction.  Suppose that $f_n$ has an obstruction
$\zd$ with slope $t$.  We maintain the notation $c(t)=c(\zd)$ and
$d(t)=d(\zd)$ from early in Section~\ref{sec:nmlip}.  By assumption,
$\zm_{f_n}(t)=t$ and $c(t)/d(t)\ge 1$, where $\zm_{f_n}$ is the slope
function of $f_n$.  Although $\zm_{f_n}(0)=0$, it is clear that
$c(0)=1$ and $d(0)=n$.  Hence $t\ne 0$.

Now we apply Theorem~\ref{thm:omit1}.  We take $\za$ there to be the
image in $S^2$ of the line segment joining $(1,0)$ and $(2,0)$, so
that $\za$ has slope $s=0$.  We may, and do, take $\widetilde{\za}=\za$.
Then $d(\widetilde{\za})=1$.  Theorem~\ref{thm:omit1} implies that
$c(\zd)=1$.  Moreover, $d(\zd)=1$ because $c(\zd)/d(\zd)\ge 1$.

Now we apply Theorem~\ref{thm:nmlip3} with the same $\za$ and $\zd$.
We easily conclude that $\zi(\za,\zd)\le 2$.  So if we write
$t=\frac{p}{q}$, where $p$ and $q$ are relatively prime integers, then
$\left|p\right|\le 2$.  We next apply Theorem~\ref{thm:nmlip3} with
$\za$ there equal to the image of the line segment with slope 2
joining $(-1,0)$ and $(0,2)$.  We again may take $\widetilde{\za}=\za$
with $d(\za)=1$.  As before, $\zi(\za,\zd)\le 2$.  Thus
$\left|p-2q\right|\le 2$.  The inequalities $\left|p\right|\le 2$ and
$\left|p-2q\right|\le 2$ imply that $t\in \{0,1,2,\infty \}$.

Now we apply Theorem 4.1 of \cite{cfpp}.  It implies that $d(t)$ is
the order of the image of $(q,p)$ in the group $\bbZ^2/\zL_1$,
where
  \begin{equation*}
\zL_1=\{a(n,0)+b(-1,1):a,b\in \bbZ\}=\{(a,b)\in \bbZ^2:a+b\equiv
0\text{ mod } n\}.
  \end{equation*}
So the condition $d(t)=1$ implies that $(q,p)\in \zL_1$.  Since $n\ge
4$, this and the last paragraph imply that $f_n$ has no obstruction
and is therefore Thurston equivalent to a rational map.

We next begin to prove the statement about mating equators.  We will
do this by applying a criterion in Theorem 4.8 of Meyer's paper
\cite{m}.  This criterion goes back, at least in degree 2, to Theorem
7.2.1 of Wittner's thesis \cite{w}, where it is attributed to
W. Thurston.  It implies that formal matings arise from equators.  Let
$P_{f_n}$ be the postcritical set of $f_n$.  An equator for $f_n$ is a
simple closed curve $\zg$ in $S^2\setminus P_{f_n}$ which is neither
peripheral nor null homotopic such that $f_n^{-1}(\zg)$ is a simple
closed curve homotopic to $\zg$ relative to $P_{f_n}$ which $f$ maps
to $\zg$ in an orientation-preserving manner.  The condition that
$f_n^{-1}(\zg)$ is a simple closed curve is equivalent to the
condition that $d(s)=n$, where $s$ is the slope of $\zg$.  The
condition that $f_n^{-1}(\zg)$ is homotopic to $\zg$ relative to
$P_{f_n}$ is equivalent to the condition that $\zm_{f_n}(s)=s$.  The
orientation condition is satisfied if $f_n$ fixes an element of
$P_{f_n}$.  So to prove that a simple closed curve $\zg$ in
$S^2\setminus P_{f_n}$ with slope $s$ is an equator, it suffices to
prove that $d(s)=n$, $\zm_{f_n}(s)=s$ and that $f_n$ fixes an element
of $P_{f_n}$.

The last of these requirements is the easiest to check.  As for
general NET maps, we have the group of Euclidean isometries
  \begin{equation*}
\zG_1=\{x\mapsto 2\zl\pm x:\zl\in \zL_1\}.
  \end{equation*}
A straightforward computation shows that $f_n$ fixes all of its
postcritical points except the image of $(n-1,1)$ in $\bbR^2/\zG_1$
when $n$ is even.  When $n$ is odd, $f_n$ also fixes the image of
$(n-1,1)$ in $\bbR^2/\zG_1$.

For the remaining requirements, let $m$ be an integer with $0\le m\le
\left\lceil \frac{n-2}{2}\right\rceil $.  Let $p=2m$ and let
$q=n-2m-1$.  Setting $s=\frac{p}{q}$, we will show that $d(s)=n$ and
$\zm_{f_n}(s)=s$.

We show that $d(s)=n$ in this paragraph.  Let $p'$ and $q'$ be
relatively prime integers such that $\frac{p'}{q'}=\frac{p}{q}$.  As
in four paragraphs above, Theorem 4.1 of \cite{cfpp} implies that
$d(s)$ is the least positive integer such that $d(s)(p'+q')\equiv
0\text{ mod } n$.  But $\gcd(p'+q',n)=1$ because $p'+q'$ divides
$p+q=n-1$.  Thus $d(s)=n$.

We finally prove that $\zm_{f_n}(s)=s$.  To do this, we use the visual
interpretation of slope function evaluation in Section~\ref{sec:eval}.
So we have a line segment $S$ in $\bbR^2$ with slope $s$ and endpoints
$v$ and $w$ such that $S$ maps to a simple closed curve in
$\bbR^2/\zG_1$.  We assume that $S$ contains no elements of $\bbZ^2$.
We are led to the action of $2\zL_1$ on $\bbR^2$.  The closed
parallelogram with corners $(0,0)$, $(2n,0)$, $(-1,1)$ and $(2n-1,1)$
is a fundamental domain for $\zG_1$.  Doubling this, we obtain the
closed parallelogram with corners $(0,0)$, $(2n,0)$, $(-2,2)$ and
$(2n-2,2)$, which is a fundamental domain for $2\zL_1$.  The
translates of the latter parallelogram by elements of $2\zL_1$ form a
tesselation $\cT$ of $\bbR^2$.

In the visual interpretation of the computation of $\zm_{f_n}(s)$, a
photon starts at $v$ and traverses an initial segment of $S$ until it
hits a spin mirror.  Our nontrivial spin mirrors have two sizes, small
and large.  Whether small or large, the photon spins about the center
of this spin mirror.  Then it travels parallel to $S$ but in the
direction opposite to the direction that it was traveling just before
it hit the spin mirror.  It continues in this way.  Let $P$ be the
photon's path, the set of points traversed by the photon.  There
exists a canonical map $\zf\co P\to S$.

To analyze $P$, we use the basis $B$ of $\bbR^2$ consisting of $(1,0)$
and $(-1,1)$.  We use the notation $(x,y)_B$ to denote the point
$x(1,0)+y(-1,1)\in \bbR^2$.  So for every $(x,y)\in \bbR^2$, we have
that $(x,y)=(x+y,y)_B$.  Since the vector $(q,p)$ gives the direction
of $S$, the slope of $S$ with respect to $B$ is
$\frac{p}{p+q}=\frac{p}{n-1}$.

We put conditions on $v$ and $w$ in this paragraph.  It is easy to see
that $\zm_f(0)=0$.  So we assume that $s\ne 0$.  The line segment
joining $(1,0)$ and $(2,0)$ in $\bbR^2$ maps to a core arc for
$(S^2,P_{f_n})$ with slope 0.  Since $s\ne 0$, every simple closed
curve in $S^2\setminus P_f$ with slope $s$ meets this core arc in its
interior.  It follows that we may choose $v$ so that $v=(x,0)_B$ with
$1<x<2$.  We choose $w$ so that its second $B$-coordinate is positive.

We say that real numbers $x$ and $y$ are congruent modulo an integer
$k$ if and only if $x-y\in k \bbZ$.  If $u\in P$ and $\zf(u)$ are
moving in the same direction, then their first $B$-coordinates are
congruent modulo $2n$ and their second $B$-coordinates are congruent
modulo 4.  If $u$ and $\zf(u)$ are moving in opposite directions, then
their first $B$-coordinates are congruent to minus each other modulo
$2n$ and their second $B$-coordinates are congruent to minus each
other modulo 4.  In particular, a point $u\in P$ is in a spin mirror
if and only if $\zv(u)$ is in a spin mirror.

We are now ready to analyze $P$.  The analysis separates into two
cases.
\smallskip

\noindent\textsl{Case 1.} The photon enters a parallelogram $T\in \cT$
at a point $u_0\in P$ which is in the gap between spin mirrors in the
lower left corner of $T$.  This is the case when $u_0=v$.  This case
is illustrated in the top diagram of Figure~\ref{fig:prlgs}.  For both
of the diagrams in Figure~\ref{fig:prlgs}, the photon enters the
parallelogram at $A$, travels to $B$, spins to $B'$ and continues in
this way until it exits the parallelogram at either $G$ or $H$.  Let
$u\in P$ be the photon's position, and suppose that $\zf(u)=(x,y)_B$.
As $\zf(u)$ moves toward $w$ from $\zf(u_0)=(x_0,y_0)_B$, the photon
remains in $T$ until one of three conditions is met.  One condition is
that $y\in 2\bbZ$ and $\pm x$ is congruent modulo $2n$ to a real
number between 0 and 1.  Another condition is that $y\in 2\bbZ$ and
$\pm x$ is congruent modulo $2n$ to a real number between 1 and 2.
The remaining condition is that $x$ is congruent to 0 modulo $2n$.  If
the first condition is met, then the photon leaves $T$ by spinning
about a small spin mirror.  If the second condition is met, then the
photon leaves $T$ through either its top or bottom between two spin
mirrors.  If the third condition is met, then the photon leaves $T$
through a side of $T$.

Consider what happens when the second $B$-coordinate of $\zf(u)$
increases from $y_0$ by a nonnegative real number $r$.  Since the
slope of $S$ with respect to $B$ is $\frac{p}{n-1}$, as the second
$B$-coordinate of $\zf(u)$ increases by $r$, its first $B$-coordinate
increases by $\frac{n-1}{p}r$.  We have that $1\le p\le 2\left\lceil
\frac{n-2}{2}\right\rceil\le n-1$, so $\frac{n-1}{p}\ge 1$.  In
particular, when $r=2$ the first $B$-coordinate of $\zf(u)$ is at
least 3.  So the photon does not leave $T$ through the gap between the
spin mirrors in the upper left corner of $T$ when $r=2$.  Furthermore,
if $r\le 2p-2$, then
  \begin{equation*}
x=x_0+\frac{n-1}{p}r\le x_0+2(n-1)-2\frac{n-1}{p}\le x_0+2n-4.
  \end{equation*}
So for such values of $r$, the photon remains in $T$.  Because $p=2m$,
we have that $2p-2\equiv 2\text{ mod } 4$.  Hence the photon is in the
top of $T$ when $r=2p-2$.  It is in the large spin mirror in the top
of $T$.  When $r=2p$, we have that $x=x_0+2n-2$.  So the photon is in
a small spin mirror in the bottom of $T$.  This spin mirror must be
the one in the lower left corner of $T$, as in the top diagram in
Figure~\ref{fig:prlgs}.  The photon then leaves $T$ by spinning about
this spin mirror.  It quickly re-enters $T$ through the left side of
$T$.  Now there are essentially two possibilities.  It might proceed
directly to the gap between the spin mirrors in the upper left corner
of $T$.  We are back in Case 1.  Otherwise, we note that $p=2m$,
$q=n-2m-1$ and $\frac{p}{q}=s$.  We have dismissed the case in which
$s=0$.  So $p\ge 2$ and $q\le n-3$, which implies that
$\frac{p}{q}>\frac{2}{n-2}$.  Hence if we do not return to Case 1,
then the photon proceeds directly to the left half of the large spin
mirror in the top of $T$.  This puts us in Case 2.
\smallskip

  \begin{figure}
\centerline{\includegraphics{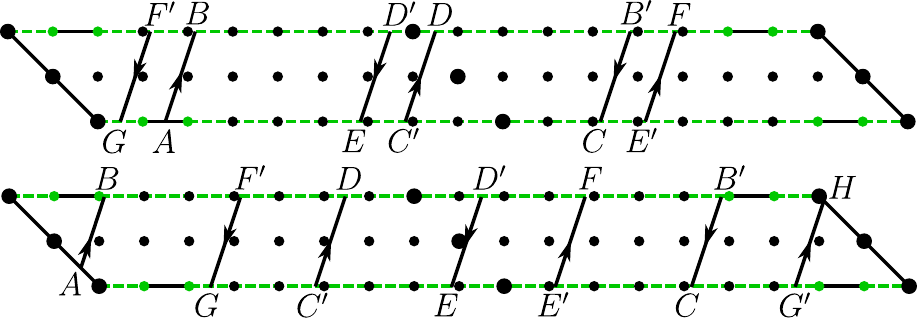}}
\caption{The two main photon behaviors within parallelograms}
\label{fig:prlgs}
  \end{figure}

\textsl{Case 2.} The photon enters a parallelogram $T\in \cT$ through
its left side and proceeds directly to the left half of the large spin
mirror in the top of $T$.  This case is illustrated in the second
diagram in Figure~\ref{fig:prlgs}.  Suppose that the photon enters $T$
at the point $u_1$.  Suppose that $\zf(u_1)=(x_1,y_1)_B$.  Then $x_1$
is congruent to 0 modulo $2n$.  From $u_1$ the photon proceeds
directly to a point $u_2$ in the top of $T$.  Suppose that
$\zf(u_2)=(x_2,y_2)_B$.  Then $x_2$ is congruent modulo $2n$ to a real
number between 2 and $n$.  We argue as in Case 1.  We allow the second
$B$-coordinate of $\zf(u)$ to increase by a nonnegative real number
$r$ beyond $y_1$.  If $r \le 2p$, then $x\le x_1+2n-2$.  So the photon
remains in $T$ as $r$ increases from 0 to $2p$.  Because $2p\equiv0
\text{ mod } 4$, the photon is moving up and right as $r$ approaches
$2p$.  It is near the right side of $T$.  As $r$ increases from $2p$
to $2p+2$, the number $x$ assumes the value $x_2+2n-2>2n$.  So the
photon passes out of the right side of $T$ into a parallelogram $T'$,
as in the second diagram in Figure~\ref{fig:prlgs}.

Now there are essentially three possibilities.  The photon might
proceed directly to the large spin mirror in the top of $T'$.  This
puts us back in Case 2.  It might proceed directly to the gap between
the spin mirrors in the upper left corner of $T'$.  This puts us back
in Case 1.  Finally, it might proceed directly to the small spin
mirror in the upper left corner of $T'$, as in the second diagram in
Figure~\ref{fig:prlgs}.  It then spins and re-enters $T$.  As it
continues from here, its behavior is opposite to the behavior in Case
1.  Hence it leaves $T$ through the gap between the spin mirrors in
the upper right corner of $T$.  After the photon passes through this
gap, we quickly return to the situation at the beginning of this
paragraph.  This concludes Case 2.
\smallskip

The above analysis of $P$ has the following consequence.  The photon
never leaves a parallelogram $T\in \cT$ by passing through the left
side or a gap in the bottom of $T$.  However, these parallelograms
have the defect that the photon might spin from a parallelogram $T$ to
the parallelogram immediately left of $T$.  This defect can be
repaired by replacing the parallelograms with rectangles, as in
Figure~\ref{fig:rcts}.  The midpoints of the sides of the rectangles
equal the midpoints of the sides of the parallelograms.  For these
rectangles it can be said that once the photon enters a rectangle $R$,
it is never later in the rectangle immediately left of $R$ or a
rectangle immediately below $R$.

  \begin{figure}
\centerline{\includegraphics{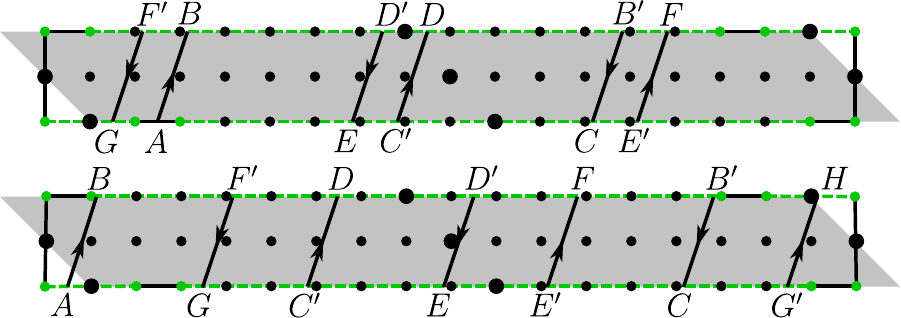}}
\caption{Replacing the shaded parallelograms with rectangles}
\label{fig:rcts}
  \end{figure}

We are finally prepared to evaluate $\zm_{f_n}(s)$.  The photon begins
at $v$.  Suppose that it ends at $w'$.  Then
$\zm_{f_n}(s)=\frac{a}{b}$, where $a$ and $b$ are the integers such
that $w'-v=b(n,0)+a(-1,1)$.  We have that $s=\frac{p'}{q'}$, where
$p'$ and $q'$ are relatively prime nonnegative integers.  The photon
moves in line segments with slope $\frac{p'}{q'}$.  The integer $a$ is
2 times the number of times that the photon passes through the top of
a rectangle.  The number of times that the photon passes through the
top of a rectangle is the number of times that the photon passes
through a gap between spin mirrors.  This is the number of times that
the line segment $S$ passes through a gap between spin mirrors.  The
image of $S$ in $\bbR^2/\zG_1$ is a simple closed curve which $f$ maps
to a simple closed curve with slope $s$.  The gap between two spin
mirrors is spanned by a line segment whose image in $\bbR^2/\zG_1$ is
a core arc with slope 0 which $f$ maps with degree 1 to a core arc
with slope 0.  It follows that $a=2\zi(s,0)=2p'$.  Similarly,
$b=2\zi(s,\infty )=2q'$.  Therefore $\zm_{f_n}(s)=\frac{2p'}{2q'}=s$.

This completes the proof of Theorem~\ref{thm:matings}.

\end{proof}

\begin{remark}\label{remark:matings} Theorem~\ref{thm:matings}
provides examples of NET maps with many formal matings.  These NET
maps have translation term $\zl_1=(n,0)$.  This choice was made to
satisfy the orientation condition for equators.  The other equator
conditions are satisfied by the NET maps with the other translation
terms.  So if we compose one of these other NET maps with itself, we
obtain another Thurston map with many formal matings.  Two simple
computations show that the NET maps with translation terms
$\zl_2=(-1,1)$ and $\zl_1+\zl_2=(n-1,1)$ are hyperbolic.  So Meyer's
Theorem 4.2 in \cite{m} implies that these NET map iterates even have
many topological matings.
\end{remark}

\end{document}